\documentclass{amsart}
\usepackage{amssymb}
\usepackage{xy}
\xyoption{arrow}
\xyoption{curve}
\xyoption{matrix}
\xyoption{cmtip}
\SelectTips{cm}{}

\DeclareFontFamily{OMS}{rsfs}{\skewchar\font'60}
\DeclareFontShape{OMS}{rsfs}{m}{n}{<-5>rsfs5 <5-7>rsfs7 <7->rsfs10 }{}
\DeclareSymbolFont{rsfs}{OMS}{rsfs}{m}{n}
\DeclareSymbolFontAlphabet{\scr}{rsfs}

\newcommand{\sI}{\scr{I}}

\newcommand{\phat}{^{\scriptscriptstyle\wedge}_{p}}
\newcommand{\twohat}{^{\scriptscriptstyle\wedge}_{2}}

\let\sma\wedge
\renewcommand{\to}{\mathchoice{\longrightarrow}{\rightarrow}{\rightarrow}{\rightarrow}}

\newcommand{\Z}{\ensuremath{\mathbb{Z}}}
\newcommand{\R}{\ensuremath{\mathbb{R}}}
\newcommand{\monoto}{\lhook\joinrel\relbar\joinrel\rightarrow}

\let\catsymbfont\mathcal 
\newcommand{\aA}{{\catsymbfont{A}}}
\newcommand{\aB}{{\catsymbfont{B}}}
\newcommand{\aC}{{\catsymbfont{C}}}
\newcommand{\aD}{{\catsymbfont{D}}}

\newcommand{\aO}{{\catsymbfont{O}}}

\newcommand{\aS}{{\catsymbfont{S}}}
\newcommand{\aT}{{\catsymbfont{T}}}
\newcommand{\aU}{{\catsymbfont{U}}}
\newcommand{\aV}{{\catsymbfont{V}}}
\newcommand{\aW}{{\catsymbfont{W}}}

\newcommand{\bA}{\mathbb{A}}
\newcommand{\bB}{\mathbb{B}}
\newcommand{\bC}{\mathbb{C}}
\newcommand{\bD}{\mathbb{D}}
\newcommand{\bE}{\mathbb{E}}

\newcommand{\bK}{\mathbb{K}}
\newcommand{\bL}{\mathbb{L}}
\newcommand{\bM}{\mathbb{M}}

\newcommand{\bO}{\mathbb{O}}

\newcommand{\bU}{\mathbb{U}}

\newcommand{\bZ}{\mathbb{Z}}

\newcommand{\htp}{\simeq}    % homotopy symbol
\newcommand{\thp}{\ltimes}   % twisted half-smash product

\DeclareFontFamily{OMS}{rsfs}{\skewchar\font'60}
\DeclareFontShape{OMS}{rsfs}{m}{n}{<-5>rsfs5 <5-7>rsfs7 <7->rsfs10 }{}
\DeclareSymbolFont{rsfs}{OMS}{rsfs}{m}{n}
\DeclareSymbolFontAlphabet{\scr}{rsfs}

\def\quickop#1{\expandafter\DeclareMathOperator\csname
#1\endcsname{#1}}
\quickop{id}\quickop{Id}\quickop{colim}\quickop{hocolim}\quickop{op}
\quickop{co}\quickop{Ar}\quickop{sing}\quickop{Hom}

\newtheorem{thm}{Theorem}[section]
\newtheorem{cor}[thm]{Corollary}
\newtheorem{lem}[thm]{Lemma}
\newtheorem{prop}[thm]{Proposition}

\newtheorem*{main-l}{Localization Theorem}
\newtheorem*{main-d}{D\'{e}vissage Theorem}
\theoremstyle{definition}
\newtheorem{defn}[thm]{Definition}
\theoremstyle{remark}
\newtheorem{rem}[thm]{Remark}

\newcommand{\sL}{\scr{L}}

\newcommand{\cyc}{\ensuremath{\operatorname{cyc}}}
\newcommand{\tor}{\ensuremath{\operatorname{Tor}}}

\input xy
\xyoption{arrow}
\xyoption{curve}
\xyoption{matrix}
\xyoption{cmtip}
\SelectTips{cm}{}
\newdir{ >}{{}*!/-5pt/\dir{>}}

\begin{document}

\title[THH of Thom spectra which are $E_\infty$ ring spectra]%
{Topological Hochschild homology of Thom spectra which are
  $E_\infty$-ring spectra}
\author{Andrew J. Blumberg}
\address{Department of Mathematics, Stanford University, 
Stanford, CA 94305}
\email{blumberg@math.stanford.edu}

\begin{abstract}
We identify the topological Hochschild homology ($THH$) of the Thom
spectrum associated to an $E_\infty$ classifying map $X \rightarrow
BG$, for $G$ an appropriate group or monoid (e.g. $U$, $O$, and $F$).
We deduce the comparison from the observation of McClure, Schwanzl,
and Vogt that $THH$ of a cofibrant commutative $S$-algebra ($E_\infty$
ring spectrum) $R$ can be described as an indexed colimit together
with a verification that the Lewis-May operadic Thom spectrum functor
preserves indexed colimits.  We prove a splitting result $THH(Mf) \htp
Mf \sma BX_+$ which yields a convenient description of $THH(MU)$.  This
splitting holds even when the classifying map $f \colon X \rightarrow
BG$ is only a homotopy commutative $A_\infty$ map, provided that the
induced multiplication on $Mf$ extends to an $E_\infty$ ring
structure; this permits us to recover Bokstedt's calculation of
$THH(H\Z)$.
\end{abstract}

\maketitle

\section{Introduction}
The algebraic $K$-theory of ring spectra encodes subtle and
interesting invariants.  It has long been known that the $K$-theory of
ordinary rings contains a great deal of arithmetic information.  On
the other hand, Waldhausen showed that there is a deep connection
between the $K$-theory of the sphere spectrum and the geometry of
high-dimensional manifolds (as seen by pseudo-isotopy theory)
\cite{waldhausen}.  Waldhausen's ``chromatic'' program for
analyzing $K(S)$ in terms of a chromatic tower of $K$-theory spectra
suggests a connection between these seemingly disparate bodies of
work, as such a tower can be regarded as interpolating from arithmetic
to geometry \cite{waldhausen-chrom}.  Recently, Rognes' development of
a Galois theory of $S$-algebras \cite{rognes-galois} and attendant
generalizations of classical $K$-theoretic descent
\cite{ausoni-rognes} along with Lurie's work on derived algebraic
geometry \cite{lurie} have raised the prospect of an arithmetic
theory of ring spectra, which would provide a unified viewpoint on
these phenomena.  To gain insight into the situation, examples
provided by computations of the $K$-theory of ring spectra which do
not come from ordinary rings are essential.

Of course, computation of algebraic $K$-theory tends to be extremely
difficult.  However, for connective ring spectra, algebraic $K$-theory
is in principle tractable via ``trace methods'', which relates
$K$-theory to the more computable topological Hochschild homology
($THH$) and topological cyclic homology ($TC$).  Specifically, there is a 
topological lifting of the Dennis trace to a ``cyclotomic trace'' map
\cite{bokstedt-hsiang-madsen}, and the fiber of this map is
well-understood \cite{dundas, mccarthy}.  Moreover, $TC(R)$ is built
as a certain homotopy limit of the fixed-point spectra of $THH(R)$ 
with regard to the action of subgroups of the circle, and so is
relatively computable via the methods of equivariant stable homotopy
theory.  One of the major early successes of this methodology was the
resolution of the ``K-theory Novikov conjecture'' by Bokstedt, Hsiang,
and Madsen \cite{bokstedt-hsiang-madsen}.  Central to their results
was a computation of the $TC$ and $THH$ associated to the ``group
ring'' $\Sigma^\infty (\Omega X)_+$, for a space $X$.

Thom spectra associated to multiplicative classifying maps provide a
natural generalization of the suspension spectra of monoids.
Moreover, many interesting ring spectra arise naturally as Thom
spectra.  The purpose of this paper is to provide an explicit and
conceptual description of the $THH$ of Thom spectra which are
$E_\infty$ ring spectra.  As the starting point for the calculation of
$TC$ is the determination of $THH$, this description provides
necessary input to ongoing work to understand the $TC$ and $K$-theory
of such spectra.  This paper is a companion to a joint paper with
R. Cohen and C. Schlichtkrull \cite{blumberg-cohen-schlichtkrull}
which uses somewhat different methods to study the $THH$ of Thom
spectra which are $A_\infty$ ring spectra.

The operadic approach to Thom spectra of Lewis and May
\cite[7.3]{lms}, \cite{may-quinn-ray} provides a Thom spectrum functor
$M$ which yields structured ring spectra when given suitable input.
Specifically, for suitable topological groups and monoids $G$, Lewis
constructs a Thom spectrum functor 
\[
M \colon \aT / BG  \to \aS \backslash S
\]
from the category of based spaces over $BG$ to the category
$\aS \backslash S$ of unital spectra.  Furthermore, he shows that if
$f \colon X \rightarrow BG$ is an $E_n$ map then $Mf$ is an $E_n$ ring
spectrum, where $E_n$ denotes an operad which is augmented over the
linear isometries operad $\sL$ and weakly equivalent to the little $n$-cubes
operad.  In particular, $M$ takes $E_\infty$ maps to $E_\infty$ ring
spectra.  Since $E_\infty$ ring spectra can be functorially
replaced by commutative $S$-algebras, we can regard $M$ as
restricting to a functor
\[
M \colon \aT[\sL]/BG \to \aC\aA_S.
\] 
Thus, $M$ produces output which is suitable for the construction of
$THH$.

The development of symmetric monoidal categories of spectra has made
possible direct constructions of topological Hochschild homology
($THH$) which mimic the classical algebraic descriptions of Hochschild
homology, replacing the tensor product with the smash product.  Thus for a
cofibrant $S$-algebra $R$, $THH(R)$ can be 
computed as the realization of the cyclic bar construction $N^{\cyc}
R$ with respect to the smash product, where $N^{\cyc} R$ is the the
simplicial spectrum 
\[[k] \rightarrow \underbrace{R \sma R \sma \ldots \sma R}_{k+1}\]
with the usual Hochschild structure maps \cite[9.2.1]{ekmm}.

Recall that the category of commutative $S$-algebras is enriched and
tensored over unbased spaces, and more generally has all indexed
colimits \cite[7.2.9]{ekmm}.  When $R$ is commutative, McClure,
Schwanzl, and Vogt \cite{mcclure-schwanzl-vogt} made precise an
insight of Bokstedt's that there should be a homeomorphism \[|N^{cyc}
R| \cong R \otimes S^1.\] Here $R \otimes S^1$ denotes the tensor of
the commutative $S$-algebra $R$ with the unbased space $S^1$.  Thus,
we can describe $THH(Mf)$ by studying $Mf \otimes S^1$.

The category of $\sL$-spaces is also tensored over unbased spaces, and
this induces a tensored structure on the category of $\sL$-maps
$f \colon X \rightarrow BG$.  Our first main theorem, proved in
Section~\ref{sec:thomleft}, states that the Thom spectrum functor is
compatible with the topologically tensored structures on its domain
and range categories.

\begin{thm}\label{commutation}
The Thom spectrum functor
\[
M \colon \aT[\sL]/BG \to \aC\aA_S
\]
preserves indexed colimits and in fact is a continuous left adjoint.
In particular, for an unbased space $A$ and an $\sL$-map
$X \rightarrow BG$, there is a homeomorphism 
\[M(f \otimes A) \cong Mf \otimes A.\]
\end{thm}

This theorem follows from an appropriate categorical viewpoint on the
Thom spectrum functor.  The category of $\sL$-spaces can be regarded
as the category $\aT[\bK]$ of algebras over a certain monad $\bK$ on
the category $\aT$ of based spaces.  We can utilize this description
to describe the category of $\sL$-maps $X
\rightarrow BG$ as the category $(\aT/BG) [\bK_{BG}]$ of algebras over
a closely related monad $\bK_{BG}$.  Similarly, the category of
$E_\infty$-ring spectra can be regarded as the category $(\aS
\backslash S)[\tilde{\bC}]$ of algebras over a monad $\tilde{\bC}$ on
the category $\aS \backslash S$ of unital spectra.  Each of these
categories admits the structure of a topological model category, by
which we mean a model category structure compatible with an enrichment
in spaces \cite[7.2-7.4]{ekmm}.  In particular, each of these
categories has tensors with unbased spaces.

Furthermore, work of Lewis \cite[7]{lms} describes the interaction of
$M$ with these monads.  Specifically, Lewis shows \cite[7.7.1]{lms} that
\[M \bK_{BG} f \cong \tilde{\bC} Mf\]
and moreover that in fact $M$ takes the monad $\bK_{BG}$ to the
monad $\tilde{\bC}$ (i.e. that the indicated isomorphism is suitably
compatible with the monad structure maps).  In
Section~\ref{sec:coltech}, we study this situation more generally and
prove the following result about the preservation of indexed colimits 
by induced functors on categories of monadic algebras;
Theorem~\ref{commutation} is then a straightforward consequence.

\begin{thm}\label{lifting}
Let $\aA$ and $\aB$ be categories tensored over unbased spaces, and
let $\bM_A$ be a continuous monad on $A$ and $\bM_B$ be a continuous
monad on $B$, such that $\bM_A$ and $\bM_B$ preserve reflexive
coequalizers.  Let $F \colon \aA \to \aB$ be a continuous functor such
that
\begin{itemize}
\item $F$ preserves colimits and tensors, and
\item There is an isomorphism $F \bM_A X \cong \bM_B FX$ which is compatible
with the monad structure maps.
\end{itemize}
Then $F$ restricts to a functor 
\[F_{\bM}\colon \aA[\bM_A] \to \aB[\bM_B]\]
which preserves colimits and tensors.  If $F$ is a left adjoint, then 
$F_{\bM}$ is also a left adjoint.
\end{thm}

In order to use the formula $M(f \otimes S^1) \cong Mf \otimes S^1$
provided by Theorem~\ref{commutation} to compute $THH(Mf)$, we must
first ensure that we have homotopical control over $Mf$.  Two
technical issues arise.  First, the cyclic bar construction
description of $THH(R)$ only has the correct homotopy type when the
point-set smash product $R \sma R$ represents the derived smash
product (for instance if $R$ is cofibrant as a commutative
$S$-algebra).  Second, when working over $BF$, Lewis' construction of
the Thom spectrum functor we give preserves weak equivalences only for
certain classifying maps (``good'' maps), notably Hurewicz fibrations.

We show in Section~\ref{sec:compute} that by appropriate cofibrant
replacement of $f\colon X \to BG$, we can ensure that $Mf$ is suitable
for computing the derived smash product.  The second problem can be
handled by the classical device of functorial replacement by a
Hurewicz fibration.  Unfortunately, it turns out to be complicated to
analyze the interaction of these two replacements.  In the companion
paper \cite{blumberg-cohen-schlichtkrull} we discuss the technical
details of the interaction between these processes.  In the present
context, we are able to obtain our main applications without
confronting this issue; although with the tools described herein the
next result is only practically applicable when $G$ is a group, in
which case all maps are good, the splitting in
Theorem~\ref{e2splitting} holds for $BF$ as well.

\begin{cor}
Let $f \colon X \rightarrow BG$ be a good map of $\sL$-spaces such
that $X$ is a cofibrant $\sL$-space.  Then $THH(Mf)$ and $M(f \otimes
S^1)$ are isomorphic in the derived category.
\end{cor}

Just as $R \otimes S^1$ is the cyclic bar construction in the category
of commutative $S$-algebras, for an $\sL$-space $X$ we can similarly
regard $X \otimes S^1$ as a cyclic bar construction
\cite[6.7]{basterra-mandell}.  Unlike commutative $S$-algebras,
$\sL$-spaces are tensored over based spaces and the tensor with an
unbased space is constructed by adjoining a disjoint basepoint.  Thus,
for an $\sL$-space $X$ it is preferable to think of the unbased tensor
$X \otimes S^1$ as the based tensor $X \otimes S^1_+$.  This
description allows us to construct a natural map to the free loop
space 
\[X \otimes S^1_+ \to L(X \otimes S^1)\]
which is a weak equivalence when $X$ is group-like.  Note that the
based tensor $X \otimes S^1$ is a model of the classifying space of
$X$, so that we have recovered the familiar relationship between
$N^{\cyc} X$ and $L(BX)$ \cite{bokstedt-hsiang-madsen}.  Furthermore,
in Section~\ref{sec:split} we use the stable splitting of $S^1_+$ to
provide an extremely useful splitting of $THH(Mf)$.

\begin{thm}\label{einfsplittingthm}
Let $f \colon X \rightarrow BG$ be a good map of $\sL$-spaces such
that $X$ is a cofibrant and group-like $\sL$-space.  Then there is a
weak equivalence of commutative $S$-algebras 
\[THH(Mf) \htp Mf \sma BX_+.\]
\end{thm}

This theorem provides convenient formulas describing $THH$ for various
bordism spectra, notably \[THH(MU) \htp MU \sma BBU_+.\]  Furthermore,
we show that this splitting theorem holds when $f \colon X \rightarrow
BG$ is only an $E_2$ map, provided that the induced multiplicative
structure on $Mf$ ``extends to'' an $E_\infty$-structure.  In this
context, the result follows from a separate analysis which exploits
the multiplicative equivalence 
\[Mf \sma Mf \htp Mf \sma X_+\]
induced by the Thom isomorphism.  Note that in the statement of the
following theorem we do not require $X$ to be cofibrant.

\begin{thm}\label{e2splitting}
Let $\aC_2$ denote an $E_2$-operad augmented over the linear
isometries operad, and let $f \colon X \rightarrow BG$ be a good
$\aC_2$ map such that $X$ is group-like.  Assume there is a map
$\gamma \colon Mf \rightarrow M^\prime$ which is a weak equivalence of
homotopy commutative $S$-algebras such that $M^\prime$ is a
commutative $S$-algebra.  Then there is a weak equivalence of
$S$-modules 
\[THH(Mf) \htp Mf \sma BX_+.\]
\end{thm}

Although the hypotheses of this theorem may seem strange, in fact this
situation arises in nature.  It has long been known that $H\Z / 2$ is
the Thom spectrum of an $E_2$ map $f \colon \Omega^2 S^3 \rightarrow BO$
\cite{cohen-may-taylor, mahowald}.  There is a similar construction of
$H\Z / p$ for odd primes due to Hopkins which is described in
\cite{thomified}.  Constructions of $H\Z$ as a Thom spectrum over
$\Omega^2 S^3\left<3\right>$ are also well-known
\cite{cohen-may-taylor, mahowald}, but these descriptions only yield
an $H$-space structure on $H\Z$.  

In Section~\ref{S:TEM}, we discuss a construction of $H\Z$ as the
Thom spectrum associated to an $E_2$ map.  Then
Theorem~\ref{e2splitting} allows us to recover the classical
computations of Bokstedt of $THH(\Z/2)$, $THH(\Z/p)$, and $THH(\Z)$.

\bigskip
These results appeared as part of the author's 2005 University of
Chicago thesis.  I would like to thank Peter May for his support and
suggestions throughout the conduct of this research.  I would also
like to express my gratitude to Michael Mandell --- this paper could
not have been written without his generous assistance.  In addition, I
would like thank Christian Schlichtkrull and Ralph Cohen for agreeing
to join forces in the preparation
of \cite{blumberg-cohen-schlichtkrull}.  The paper was improved by
comments from Christopher Douglas and Halvard Fausk.
     
\section{Colimit-preserving functors in categories of monadic
algebras}
\label{sec:coltech}

In this section, we prove Theorem~\ref{lifting}.  The theorem is
essentially a straightforward consequence of categorical results due
to Kelly describing the construction of colimits and indexed colimits
in enriched categories of monadic algebras.  We begin by reviewing the
relevant background material, largely following the exposition
of \cite{ekmm}.

Let $\aV$ denote a symmetric monoidal category, and let $\aC$ be a
category enriched over $\aV$.  In such a context we can define tensors
and cotensors (and more generally indexed colimits and limits).

\begin{defn}
Let $\aC$ be a category enriched over $\aV$.  Then $\aC$ is tensored
if there exists a functor $\otimes_{\aC} \colon \aC \times \aV \rightarrow
\aC$ which is continuous in each variable and such that there is an
isomorphism
\[\aC(X \otimes_{\aC} A, Y) \cong \aV(A, \aC(X,Y))\]
of objects of $\aA$.  There is a dual notion of cotensors.
\end{defn}

For example, both the category of based spaces and the category of
spectra are tensored over based spaces.  The tensor of a spectrum $X$
and a based space $A$ is $X \sma A$.  The cotensor of a spectrum $X$
and a based space $A$ is the mapping spectrum $F(A,X)$.  Notice that
we can define the tensor of a spectrum $X$ and an unbased space $B$ by
adjoining a disjoint basepoint to $B$ and taking the tensor with
respect to the enrichment in based spaces --- the tensor of a spectrum
$X$ and an unbased space $B$ is $X \sma B_+$.

In an enriched category, there are notions of indexed colimits and
limits which take the enrichment into account.  Tensors and cotensors
are examples of such indexed colimits and limits, and in the
topological setting are particularly important as a consequence of the
following result of Kelly \cite[7.2.6]{ekmm}.

\begin{thm}\label{t:kelly}
A topological category has all indexed colimits provided that it is
cocomplete and tensored.  Dually, a topological category has all
indexed limits provided it is complete and cotensored.
\end{thm}

For our application, we will need to understand the tensor in the
category of commutative $S$-algebras and the tensor in the category of
$E_\infty$ spaces.  A priori, it is not clear that either of these
categories is tensored.  Unlike in the case of spectra, there is not a
familiar construction which yields the tensor.  For that matter,
construction of colimits in these categories is not obvious either.
The key observation is that each of these categories can be regarded
as a category of algebras over a monad. 

Let $\bA \colon \aC \rightarrow \aC$ be a monad with multiplication $\mu$
and unit $\eta$.  Recall that an object $X$ in $\aC$ is an algebra
over $\bA$ if there is an action map $\psi \colon \bA X \rightarrow X$ such
that the following diagrams commute :
\[
\xymatrix{
\relax\bA \relax\bA X \ar[r]^-{\psi} \ar[d]^-{\mu} & \relax\bA X \ar[d]^-{\psi} & X \ar[r]^-{\eta} \ar[dr]^-{=} & \relax\bM X \ar[d]^-{\mu} \\
\relax\bA X \ar[r]^-{\psi} & X & & X\\
}.
\]

The category of commutative $S$-algebras is precisely the category of
algebras over a certain monad in $S$-modules, and the category of
$\sL$-spaces is the category of algebras over a certain monad in
based spaces; we will define these monads in Section~\ref{S:einf}.

A key observation of McClure and Hopkins \cite{hopkins}, further
developed in \cite{ekmm}, is that there are general constructions for
lifting colimits and tensors from a category $\aC$ to the category
$\aC[\bA]$ of algebras for a monad $\bA$ on $\aC$.  That is, colimits
and tensors in $\aC[\bA]$ can be constructed in terms of certain
colimits and tensors in $\aC$.  However, in order to utilize these
results a technical condition must be satisfied by the monad $\bA$,
which we will now recall \cite[2.6.5]{ekmm}.

\begin{defn}
Let $A$, $B$, and $C$ be objects of a category $\aC$.  A reflexive
coequalizer is a coequalizer diagram
\[
\xymatrix{
A \ar@<1ex>[r]^-e \ar@<-1ex>[r]_{f} & B \ar[r]^-g & C \\
}
\]
such that there exists a splitting map $h \colon B \rightarrow A$ such that
$e \circ h = \id$ and $f \circ h = \id$.
\end{defn}

In order for the lifting results to apply, $\bA$ must preserve
reflexive coequalizers.  In this situation, if $A$ and $B$ are
$\bA$-algebras, there is a unique structure of $\bA$-algebra on $C$
and moreover $C$ is the coequalizer of $A$ and $B$ in the category
$\aC[\bA]$ \cite[2.6.6]{ekmm}.  That is, we can form the coequalizer
in the category $\aC[\bA]$ by taking the coequalizer in $\aC$.  Now we
can state the lifting results.  Recall the following proposition from
EKMM \cite[2.7.4]{ekmm}.

\begin{prop}
Let $\bE$ be a continuous monad defined on a topologically enriched
category $\aC$.  If $\bE$ preserves reflexive coequalizers, then the
colimit in the category $\aC[\bE]$ of algebras over $\bE$ is given by
the following coequalizer :
\[
\xymatrix{
\relax\bE (\colim \bE R_i) \ar@<1ex>[rr]^{\bE(\colim \xi_i)} \ar@<-0.5ex>[rr]_{\mu \circ \bE \alpha} && \relax\bE(\colim R_i)
}
.
\]
Here $\mu$ is the composition map for the monad $\bE$, $\xi_i$ is the
action map $\bE R_i \rightarrow R_i$, and 
\[\alpha \colon \colim \bE R_i \rightarrow \bE \colim R_i\]
is obtained as follows.  For each $i$ there is a natural map $\iota_i
\colon R_i \rightarrow \colim R_i$, and $\alpha$ is specified as the unique
map whose composite with the natural map $\bE R_i \rightarrow
\colim \bE R_i$ is precisely $\bE$ applied to $\iota_i$.
The splitting of the coequalizer is obtained from the unit of the
monad. 
\end{prop}

There is a related technique for constructing tensors as appropriate
coequalizer diagrams via the following proposition from EKMM
\cite[7.2.10]{ekmm}.

\begin{prop}
Let $\bE$ be a continuous monad defined on a topologically enriched
category $\aC$.  If $\bE$ preserves reflexive coequalizers, then the
tensor in the category $\aC[\bE]$ of algebras over $\bE$ is given by
the following coequalizer : 

\[
\xymatrix{
\relax\bE (\bE X \otimes A) \ar@<1ex>[rr]^{\bE(\xi \otimes \id)} \ar@<-0.5ex>[rr]_{\mu \circ \bE \nu} && \relax\bE (X \otimes A)
}
,
\]
where $\nu \colon \bE X \otimes A \rightarrow \bE(X \otimes A)$ is the adjoint of composite
\[A \rightarrow \aC(X, X \otimes A) \rightarrow \aC(\bE X, \bE(X \otimes A)).\]
Here the first arrow is the adjoint of the identity map.
\end{prop}

For our application, we note that the relevant monads preserve
reflexive coequalizers and so the preceding theorems construct the
tensors and colimits in the category of commutative $S$-algebras and
the category of $E_\infty$-spaces.  The limits and cotensors are
inherited from the base categories of $S$-modules (and hence spectra)
and based spaces respectively.

We are now ready to prove Theorem~\ref{lifting}.  Let
$F \colon \aC \rightarrow \aD$ be a functor between topological
categories, let $\bA \colon \aC \rightarrow \aC$ be a monad on $\aC$,
and let $\bB \colon \aD \rightarrow \aD$ be a monad on $\aD$.  The
following easy lemma provides a simple condition for $F$ to yield a
functor on the associated categories of algebras,
$F \colon \aC[\bA] \to \aD[\bB]$.

\begin{lem}\label{l:monad-comm}
Let $\phi \colon  \bB F(X) \cong F(\bA X)$ be a natural isomorphism such
that the following diagrams commute for any object $X$ of $\aC$.
\[
\xymatrix{
\relax\bB F(X) \ar[r]^-{\phi} & F(\relax\bA X) & \relax\bB \relax\bB F(X) \ar[r]^-{\mu_B} \ar[d]^-{\phi} & \relax\bB F(X) \ar[d]^-{\phi} \\
& F(X) \ar[ul]^-{\eta_B} \ar[u]_-{F(\eta_A)} & F(\relax \bA \relax \bA X) \ar[r]^-{F(\mu_A)} & F(\bA X) \\
}
\]

Then if $X$ is a $\bA$-algebra in $\aC$ with action map $\psi \colon \bA X
\rightarrow X$, $F(X)$ is a $\bB$-algebra in $\aD$ with action map
\[
\xymatrix{
\relax\bB F(X) \cong F(\bA X) \ar[r]^-{F(\psi)} & F(X)
}
.
\]
Therefore $F$ yields a functor from $\aC[\bA]$ to $\aC[\bB]$.
\end{lem}

Now we prove the main technical result of the section.  Suppose we are
in the situation described in the preceding lemma, with the additional
assumption that $\aC$ and $\aD$ are topological categories.

\begin{thm}\label{t:pres}
Let $\aC$ and $\aD$ be cocomplete topological categories, and $\bA \colon \aC
\rightarrow \aC$ and $\bB \colon \aD \rightarrow \aD$ continuous monads.
Further suppose that there is a continuous functor $F \colon \aC
\rightarrow \aD$ which satisfies the hypothesis of the preceding lemma
and therefore yields a functor $F \colon \aC[\bA] \rightarrow \aD[\bB]$.
\begin{enumerate}
\item{If $F \colon \aC \rightarrow \aD$ preserves colimits and tensors, and the monads
$\bA$ and $\bB$ preserve reflexive coequalizers, then $F \colon \aC[\bA]
\rightarrow \aD[\bB]$ preserves colimits and tensors in $\aC[\bA]$.
Therefore $F$ preserves all indexed colimits in $\aC$.}
\item{If furthermore $F$ is a left adjoint as a functor from $\aC$ to
  $\aD$, then $F$ induces a left adjoint from $\aC[\bA]$ to $\aD[\bB]$.}
\end{enumerate}
\end{thm}

\begin{proof}
First, we handle the issue of colimits.  We can apply
\cite[2.7.4]{ekmm} to describe colimits in the category $\aC[\bA]$ of
$\bA$-algebras.  Given a diagram of $\{R_i\}$ of $\bA$-algebras, we
can describe $F(\colim R_i)$ as $F$ applied to the reflexive
coequalizer which creates colimits in the category $\aC[\bA]$.

\[
F\left(
\xymatrix{
\relax\bA(\colim \bA R_i) \ar@<1ex>[rr]^{\bA(\colim \xi_i)}
\ar@<-0.5ex>[rr]_{\mu \circ \bA \alpha} && \relax\bA(\colim R_i)
}
\right)
.
\]
Since $F$ commutes with colimits in $\bA$, this is isomorphic to the
reflexive coequalizer :

\[
\xymatrix{
\relax\bB(\colim \bB FR_i)
\ar@<1ex>[rr]^{\bB(\colim F(\xi_i))} \ar@<-0.5ex>[rr]_{\mu
  \circ \bB F(\alpha)} && \relax\bB(\colim FR_i))
}
.
\]
This is precisely the colimit of the diagram $\{FR_i\}$ in the
category of $\bB$-algebras by \cite[2.7.4]{ekmm} once again.

Next, we consider tensors.  We can express $F(X \otimes A)$ as $F$ applied to
the reflexive coequalizer which creates the tensors in the category
$\aC[\bA]$.
\[
F\left(
\xymatrix{
\relax\bA (\bA X \otimes A) \ar@<1ex>[rr]^{\bA(\xi \otimes \id)} \ar@<-0.5ex>[rr]_{\mu \circ \bA \nu} && \relax\bA (X \otimes A)
}
\right)
.
\]
We can rewrite this expression using the fact that $F$ commutes with
colimits in $\aA$, as follows.
\[
\xymatrix{
\relax\bB F(\bA X \otimes A) \ar@<1ex>[rr]^{\bB (\xi \otimes \id)} \ar@<-0.5ex>[rr]_{\mu \circ \bB \nu} && \relax\bB F(X \otimes A)
}
.
\]
As $F$ commutes with tensors in $\aA$, this becomes :
\[
\xymatrix{
\relax\bB (\bB FX \otimes A) \ar@<1ex>[rr]^{\bB (\bB (\xi \otimes \id)} \ar@<-0.5ex>[rr]_{\mu \circ \bB \nu} && \relax\bB (FX \otimes A)
}
.
\]
This is precisely the diagram expressing the tensor $FX \otimes A$
in the category $\aC[\bB]$.  It is now a consequence of theorem
\ref{t:kelly} that $M$ preserves all indexed colimits.  

Finally, assume that $F \colon \aC \rightarrow \aD$ is a left adjoint.
There is a diagram of categories:
\[
\xymatrix{
\relax\aC[\relax\bA] \ar[r]^-F \ar@<-0.5ex>[d]_U &
\relax\aD[\relax\bB] \ar@<-0.5ex>[d]_V \\ 
\relax\aC \ar@<-0.5ex>[u]_G \ar[r]^-F & \relax\aD \ar@<-0.5ex>[u]_G. \\
}
\]
Here $U$ and $V$ denote forgetful functors, and $G$ denotes the free
algebra functors.  The square commutes in the sense that $F \circ G =
G \circ F$ and $F \circ U = V \circ F$.  To show that $F \colon \aC[\bA]
\rightarrow \aD[\bB]$ is a continuous left adjoint, it suffices to
show that $F$ preserves tensors and $F$ is a left adjoint when the
enrichment is ignored \cite[6.7.6]{borceux}.  We know that the former
holds, and since $F \colon \aA \rightarrow \aB$ is a left adjoint by
hypothesis and $\aC[\bA]$ has coequalizers, we can apply the adjoint
lifting theorem \cite[4.5.6]{borceux} and conclude the latter.
\end{proof}

\section{Parameterized spaces and operadic algebras}\label{S:einf}

In this section, we review the definitions of the domain and range
categories of the Lewis-May operadic Thom spectrum functor.  We begin
by discussing operadic algebras.  

\subsection{Review of operadic algebras}

Let $\sI$ be the (unbased) category of finite-dimensional or
countably-infinite real inner product spaces and linear isometries.
This is a symmetric monoidal category under the direct sum.

\begin{defn}
Let $U^j$ be the direct sum of $j$ copies of $U$ (an
infinite-dimensional real inner product space), and let $\sL(j)$ be
the mapping space $\sI(U^j,U)$.  The action of $\Sigma_j$ on $U^j$ by
permutation induces an action of $\Sigma_j$ on $\sL(j)$.  There are
maps
\[\gamma \colon \sL(k) \times \sL(j_1) \times \ldots \times \sL(j_k) \rightarrow \sL(j_1 + \ldots + j_k)\]
given by $\gamma(g;f_1, \ldots, f_k) = g \circ (f_1 \oplus \ldots
\oplus f_k)$.  The spaces $\sL(j)$ form an operad, which we will refer
to as the linear isometries operad.
\end{defn}

The properties of the linear isometries operad have been explored at
length, notably in section XI of \cite{ekmm}.  Recall that $\sL$ is an
$E_\infty$-operad, as $\sL(j)$ is contractible, $\sL(1)$ contains the
identity, $\sL(0)$ is a point, and $\Sigma_n$ acts freely on $\sL(n)$.
We can consider both based spaces and spectra which admit actions of
$\sL$.  We will make frequent use of the fact that for any operad $\aO$,
there is an associated monad $\bO$ such that objects $X$ with actions
by $\aO$ are precisely algebras over $\bO$ \cite{may-geom}.

A space $X$ with an action of the operad $\sL$ is the same as an
algebra over a certain monad $\bK$ on the category of based spaces.
Since the monad $\bK$ preserves reflexive coequalizers, standard
lifting techniques suffice to show the following theorem
\cite{hopkins}, \cite[6.2]{basterra-mandell}.

\begin{thm}
The category $\aT[\bK]$ of $\sL$-spaces admits the structure of a
topological model category.  Fibrations and weak equivalences are
created in the category $\aT$, and cofibrations are defined as having
the left-lifting property with respect to acyclic fibrations.
\end{thm}

Since $\sL$ is an $E_\infty$ operad, we can functorially associate a
spectrum $Z$ to an $\sL$-space $X$ such that the map $X \rightarrow
\Omega^\infty Z$ is a group completion.  When $\pi_0(X)$ is a group
and not just a monoid, this map is a weak equivalence.  Such
$\sL$-spaces $X$ for which $\pi_0(X)$ is a group are said to be
group-like.

Similarly, the category of $E_\infty$-ring spectra can be described as
a category of algebras over monads, following \cite[2.4]{ekmm}.  Let
$\aS$ denote the category of coordinate-free spectra \cite{lms}.  For
clarity, we emphasize that $\aS$ is not a symmetric monoidal category
of spectra prior to passage to the homotopy category.  An
$E_\infty$-ring spectrum structured by the operad $\sL$ is an algebra
over a certain monad $\bC$ in $\aS$.

Since the Thom spectrum associated to an object $f$ of $\aT / BG$ will
have a natural unit $S \rightarrow Mf$ induced by the inclusion of the
basepoint, we also consider the category $\aS \backslash S$ of
unital spectra.  In this setting, an $E_\infty$-ring spectrum $X$ over
the operad $\sL$ is the same as an algebra over the monad
$\tilde{\bC}$, where $\tilde{\bC} X$ is a ``reduced'' version of $\bC$
quotiented to ensure that the unit provided by the algebra structure
coincides with the existing unit.

There is a close relationship between the category of algebras over
$\bC$ and algebras over $\tilde{\bC}$ \cite[2.4.9]{ekmm}.
The category $\aS \backslash S$ is itself a category of algebras over
$\aS$ for the monad $\bU$ which takes $X$ to $X \vee S$.  The monad
$\bC$ is precisely the composite monad $\tilde{\bC} \bU$, and in this
situation the categories of algebras are equivalent
\cite[2.6.1]{ekmm}.  Therefore the two notions of $E_\infty$-ring
spectrum we have described are equivalent.  In the language of
\cite{ekmm}, $\tilde{\bC}$ is the ``reduced'' monad associated to the
monad $\bC$.  Both of these monads preserve reflexive coequalizers.

Finally, given an $E_\infty$-ring spectrum, the functor $S \sma_{\sL}
-$ converts it to a weakly equivalent commutative
$S$-algebra \cite[2.3.6,2.4.2]{ekmm}.  Moreover, $S \sma_{\sL} -$ is a
continuous left adjoint.

\subsection{Parametrized operadic algebras}

Now we move on to consider the category of spaces over a fixed base
space $B$.  The category $\aU / B$ has objects maps $p \colon
X \rightarrow B$, where $X$ and $B$ are objects of $\aU$.  A morphism
$(p_1 \colon X \rightarrow B) \rightarrow (p_2 \colon Y \rightarrow
B)$ is a map $f \colon X \rightarrow Y$ such that $p_2 f = p_1$.  The
properties of this category have been investigated in a variety of
places \cite{lewis, intermont-johnson}, \cite[7.1]{lms}.  In
particular, this is a topological category where the tensor of
$p \colon X \rightarrow B$ and an unbased space $A$ is given by the
composite  
\[
\xymatrix{
X \times A \ar[r]^{\pi_1} & X \ar[r]^p & B \\
}
\]
(where $\pi_1$ is the projection onto the first factor).

Since we will be interested in spaces which admit operad actions, we
also consider the related category of based spaces over $B$.  This is
the category $\aT / B$, defined in the same fashion as $\aU / B$,
replacing spaces with based spaces and requiring that the maps be 
based.  The category $\aT / B$ inherits the structure of a category
tensored over unbased spaces from $\aU /B$, where the tensor of
$X \rightarrow B$ and an unbased space $A$ is given by $X \sma
A_+ \rightarrow B$. 

Colimits in $\aT / B$ are formed as follows.  Given a diagram $D
\rightarrow \aT / B$, via the forgetful functor we obtain a
diagram $D \rightarrow \aT / B \rightarrow \aT$.  The colimit over $D
\rightarrow \aT / B$ is computed by taking the colimit of this
diagram in $\aT$ and using the induced map to $B$ given by the
universal property of the colimit.

When $B$ is an $\sL$-space, there is a subcategory of $\aT / B$ where
the objects are $\sL$-maps $X \rightarrow B$ and the morphisms are
$\sL$-maps over $B$.  In slight abuse of terminology, we will
sometimes refer to this category as $\sL$-spaces over $B$.  We can
regard this category as algebras over a monad on $\aT / B$.  Given a
map $f \colon Y \rightarrow B$, where $B$ is an $\sL$-space, the space
$\bK Y$ admits an $\sL$-map to $B$ given by the unique extension of
$f$ \cite[7.7]{lms}.  This specifies a monad on $\aT / B$, with
structure maps inherited from those of $\bK$, which we will refer to
as $\bK_{B}$.  Denote by $(\aT / B)[\bK_{B}]$ the category of
$\bK_{B}$-algebras.  There is a model structure on this category
defined in analogy with the naive model structure on $\aT / B$.  We
need to verify the existence of tensors and colimits in $(\aT /
B)[\bK_{B}]$.  In order to show that $(\aT / B)[\bK_{B}]$ is
topologically cocomplete, it will suffice to show that the monad
$\bK_{B}$ preserves reflexive coequalizers.  This follows immediately
from the fact that $\bK$ preserves reflexive coequalizers, since
colimits in $\aT / B$ are constructed by taking the colimit in $\aT$
and using the natural map to $B$.

\begin{prop}
The category $(\aT / B)[\bK_{B}]$ is topologically cocomplete (and in
particular has all colimits and tensors with based spaces). 
\end{prop}

It will be useful later on to write out an explicit description of the
tensor in $(\aT / B)[\bK_{B}]$.  We regard the category of
$\sL$-spaces as tensored over unbased spaces via the tensor over based
spaces: for an unbased space $A$ the tensor with an $\sL$-space $X$ is
the based tensor $X \otimes A_+$.

\begin{lem}
The tensor of an unbased 
space $A$ and $(X \rightarrow B)$ 
is given by 
\[X \otimes A_+ \rightarrow X \otimes S_0 \cong X \rightarrow BG,\]
where the first map is the collapse map which takes $A$ to the
non-basepoint of $S^0$. 
\end{lem}

\section{The operadic Thom spectrum functor}

In this section we review the operadic theory of Thom spectra
developed by Lewis \cite[7.3]{lms} and May \cite{may-quinn-ray}.
Our discussion is updated slightly to take account of more recent
developments in the theory of diagram spectra \cite{mandell-may,
  mmss}.  In particular, our terminology regarding $\sI$-spaces
reflects the modern usage and is at variance with the definitions in
the original articles.

\subsection{The definition of $M$}

Recall that $\sI$ denote the category of finite-dimensional or
countably-infinite real inner product spaces and linear isometries.   

\begin{defn}
An $\sI$-space is a continuous functor $X$ from $\sI$ to the category
 of based topological spaces.
\end{defn}

We will restrict attention to $\sI$-spaces with the property that
$X(V)$ is the colimit of $X(W)$ for the finite-dimensional subspaces
$W \subset V$.  This constraint implies that it is sufficient to
consider the restriction of $X$ to the full subcategory of $\sI$
consisting of the finite-dimensional real inner product spaces
\cite[1.1.8-1.1.9]{may-quinn-ray}.   

The idea of using $\sI$ to capture structure about infinite loop
spaces and operad actions dates back to Boardman and Vogt's original
treatment \cite{boardman-vogt}.  In the context of Thom spectra,
$\sI$-spaces first arose in \cite{may-quinn-ray}.  More recently, May
has introduced the terminology of``functors with cartesian product''
(FCP) to highlight the connection to diagram spectra \cite{may-fcp},
in analogy with Bokstedt's ``functors with smash product'' (FSP's).

\begin{defn}
A functor with cartesian product over $\sI$ ($\sI$-FCP) is a
$\sI$-space equipped with a unital and associative ``Whitney sum''
natural transformation $\omega$ from $X \times X$ to $X \circ \oplus$. 
\end{defn}

A commutative $\sI$-FCP is a $\sI$-FCP for which the natural
transformation $X \times X$ to $X \circ \oplus$ is commutative.  We
will assume in the following that by default $\sI$-FCP's are
commutative.  Commutative $\sI$-FCP's encode an $E_\infty$-structure
\cite[1.1.6]{may-quinn-ray}; specifically, a commutative $\sI$-FCP $X$
yields an $\sL$-space structure on $X(\R^{\infty})$.  The essential
observation is that we can use the Whitney sum to obtain a natural map
$\sL(j) \times X(R^{\infty})^j \rightarrow X(R^{\infty})$ specified by
\[
(f, x_1, x_2, \ldots, x_j) \mapsto Xf(x_1 \oplus x_2 \oplus \ldots \oplus x_j).
\]
Similarly, a noncommutative $\sI$-FCP yields a
non-$\Sigma$ $\sL$-space structure on $X(\R^{\infty})$.

There is an obvious product structure on the category of $\sI$-spaces
specified by the levelwise cartesian product.  A monoid $\sI$-FCP is
an $\sI$-FCP such that the levelwise monoid product specifies a
morphism of $\sI$-spaces.  A notable example is the monoid $\sI$-FCP
$F$ given by taking $F(V)$ to be the space of based homotopy
equivalences of $S^V$.  We will always assume that for a monoid
$\sI$-FCP $X$, the monoids $X(V)$ are grouplike.  Analogously, we will
consider group $\sI$-FCP's.  Familiar examples include the functor
specified by $V \mapsto O(V)$ and the functor specified by $V \mapsto
U(V)$.

For any monoid $\sI$-FCP $X$, we can construct a related $\sI$-FCP $BX$
via the two-sided bar construction.  Specifically, define $BX$ as the
functor specified by
\[BX(V) = B(*,X(V),*),\]
where $B(-,-,-)$ denotes the geometric realization of the two-sided bar
construction.  When $X$ is equipped with an augmentation to $F$ which
is a map of monoid $\sI$-FCP's, we can construct $EX$ as \[EX(V) = B(*,X(V),S^V),\]
where $X(V)$ acts on $S^V$ via the augmentation.  There is a
projection map $\pi \colon EX \rightarrow BX$ and a section defined by
the basepoint inclusion $* \monoto S^V$.  This section is a
cofibration, $\pi$ is a quasifibration, and $\pi$ has fiber $S^V$
\cite[7.2]{lms}.  If $X$ actually takes values in groups, $\pi$ is a
bundle.

When $X = F$, this construction provides a model of the universal
quasifibration with spherical fibers \cite{may-class}.  More
generally, we obtain universal quasifibrations and fibrations with
spherical fibers and prescribed structure groups.  Note that we are
following Lewis in letting $EG(V)$ denote the total space of the
universal spherical quasifibration rather than the associated
principal quasifibration.

Moving on, we now describe the Thom spectrum construction.  Let $G$ be a
monoid $\sI$-FCP which is augmented over $F$.  Abusing notation, we
will write $BG$ to denote both the $\sI$-FCP $BG$ and the space
$\colim_V BG(V)$.  We will assume we are given a map of spaces $f
\colon Y \rightarrow BG$.  

\begin{defn}
Let $f\colon Y \rightarrow BG$ be a map of spaces.  The filtration of
$BG$ by inner product spaces $V$ induces a filtration on $Y$ defined
as $Y(V) = f^{-1}(BG(V))$.  The Thom prespectrum associated to $f
\colon Y \rightarrow BG$ is specified as follows.  Set $Tf(V)$ to be
the Thom space of the pullback $Z(V)$ in the diagram :
\[
\xymatrix{
Z(V) \ar[r]\ar[d] & EG(V) \ar[d] \\
Y(V) \ar[r] & BG(V). \\
}
\]
That is, the map $Z(V) \rightarrow Y(V)$ has a section $i$, and $Tf(V)
= Z(V)/i(Y(V))$.  $Tf$ is a prespectrum, and we define the Thom
spectrum in $\aS$ associated to $f$ as the spectrification $Mf = LTf$.  
\end{defn}

Other filtrations can also be used in this construction, but it can be
shown that the choice of filtration does not matter up to isomorphism
of spectra \cite[7.4.4]{lms}.

To see that $Tf$ is actually a prespectrum, we must describe the
suspension maps.  Associated to the inclusion $V \subset W$ is an
inclusion $Y(V) \subset Y(W)$, and this induces a map of pullbacks
$Q_V \rightarrow Z_W$ in the following diagrams :
\[
\xymatrix{
Z_W \ar[r] \ar[d] & EG(W) \ar[d] && Q_V \ar[r] \ar[d] & EG(V) \ar[r] \ar[d] & \ar[d] EG(W) \\
Y(W) \ar[r] & BG(W) && Y(V) \ar[r] & BG(V) \ar[r] & BG(W). \\
}
\]
Upon passage to Thom spaces, we can identify the Thom space of $Q_V$
as the fiberwise suspension $\Sigma^{W-V}$ of the Thom space of $Z_V$
\cite[7.2.2]{lms}, and so the map in question is a suspension map.
One checks that these suspension maps are appropriately coherent
\cite[7.2.1]{lms}.

\begin{rem}
Lewis treated only monoid $\sI$-FCP's $X$ augmented over $F$; this
augmentation gives an action of $X$ on $S^V$ which allows the
construction of $EX$.  However, we can develop the theory of Thom
spectra for other choices of fiber, as long as we specify a levelwise
action of $X$ on the fiber.  Such constructions will be useful for us
when considering models of Eilenberg-Mac Lane spectra as Thom spectra
in section \ref{S:TEM}.  We will consider $p$-local and $p$-complete
spherical fibrations, and employ ``localized'' and ``completed''
versions of $F$ formed from spaces of based self-equivalences of the
$p$-local sphere $S^V_{(p)}$ and based self-equivalences of the
$p$-complete sphere $(S^V)\phat$.
\end{rem}

We have constructed the Thom spectrum as a continuous functor from
$\aU / BG$ to coordinate-free spectra $\aS$.  Working with $\aT / BG$,
we obtain a functor to $\aS \backslash S$, unital spectra.  Here the
unit $S \rightarrow Mf$ is induced by the inclusion $* \rightarrow X$
over $BG$.  In abuse of notation, we will refer to both of these
functors as $M$.

\subsection{Properties of $M$}

Lewis proves that the Thom spectrum functor $M$ preserves colimits in
$\aU / BG$ \cite[7.4.3]{lms}.  It is straightforward to extend this to
the functor $M$ from $\aT / BG$ to $\aS \backslash S$.

\begin{lem}
The Thom spectrum functor takes colimits in $\aT / BG$ to colimits in
the category $\aS \backslash S$. 
\end{lem}

\begin{proof}
A colimit over $\aD$ in $\aT / BG$ is given as the pushout in $\aU / BG$ 
\[
\xymatrix{
\relax{\colim_{\aD}} * \ar[r] \ar[d] & \ar[d] \relax{*} \\
\relax{\colim_{\aD}} R_d \ar[r] & Z \\
}
\]
\\
where the indicated colimits are also taken in the category $\aU /
BG$.  Similarly, a colimit over $\aD$ in $\aS \backslash S$ is
constructed as the pushout in $\aS$ 
\[
\xymatrix{
\colim_{\aD} S \ar[r]\ar[d] & \ar[d] S \\
\colim_{\aD} R_d \ar[r] & Z \\
}
\]
\\
where the indicated colimits are also taken in $\aS$.  The result
follows from the fact that $M$ takes colimits in $\aU / BG$ to
colimits in spectra and $M(*) \cong S$. 
\end{proof}

Lewis also shows that the functor $M$ also preserves tensors with
unbased spaces in $\aT / BG$ \cite[7.4.6]{lms}. 

\begin{prop}\label{P:ten}
The Thom spectrum associated to the composition $X \sma A_+
\rightarrow X \rightarrow BG$ is naturally isomorphic to $Mf \sma
A_+$.
\end{prop}

When $A = I$, this implies that functor $M$ converts fiberwise
homotopy equivalences into homotopy equivalences in $\aS \backslash
S$.  

The question of invariance under weak equivalence is somewhat more
delicate.  Unfortunately, quasifibrations are not preserved under
pullback along arbitrary maps.  This can cause technical difficulty
when working with $BF$, or any other monoid $\sI$-FCP (which is not a
group $\sI$-FCP).  Following Lewis \cite[7.3.4]{lms}, we make the
following definition.

\begin{defn}
Define a map $f \colon X \rightarrow BG$ to be good if the projections $Z_V
\rightarrow X(V)$ are quasifibrations and the sections $X(V)
\rightarrow Z_V$ are Hurewicz cofibrations.
\end{defn}

A map $f \colon X \rightarrow BG$ associated to an $\sI$-monoid $G$ with
values in groups is always good, and all Hurewicz fibrations are good
\cite[7.3.4]{lms}.  Therefore, it is sometimes useful to replace
arbitrary maps by Hurewicz fibrations when working over $BF$ via the functor
$\Gamma$ \cite[7.1.11]{lms}.  This is compatible with the linear
isometries operad --- recall that given an $\sL$-map $f \colon
X \rightarrow BF$, the map $\Gamma f \colon \Gamma X \rightarrow BF$
is also an $\sL$-map \cite[1.8]{may-geom}.  Our discussion of $\Gamma$
is brief, as we do not use it extensively in this paper.

When the maps in question are good, the Thom spectrum functor
preserves weak equivalences over $BG$ \cite[7.4.9]{lms}.

\begin{thm}
If $f \colon X \rightarrow BG$ and $g \colon X^{\prime} \rightarrow BG$ are good
maps such that there is a weak equivalence $h \colon X \htp X^{\prime}$
over $BG$, then there is a stable equivalence $Mh \colon Mf \htp Mg$.
\end{thm}

In this situation, $M$ also takes homotopic maps to stably equivalent
spectra \cite[7.4.10]{lms}.  Note however that the stable equivalence
depends on the homotopy.

\begin{thm}
If $f \colon X \rightarrow BG$ and $g \colon X \rightarrow BG$ are good maps
which are homotopic, then there is a stable equivalence $Mf \htp Mg$.
\end{thm}

\section{The Thom spectrum functor is a left adjoint}\label{sec:thomleft}

As discussed previously, spaces with actions by the linear isometries
operad $\sL$ can be regarded as the category $\aT[\bK]$ of algebras
over the monad $\bK$.  Similarly, spectra in $\aS \backslash S$ which
are $E_\infty$-ring spectra structured by the linear isometries operad
can be regarded as the category $(\aS \backslash S)[\bC]$ of algebras
with respect to the monad $\tilde{\bC}$.

One of the main results of Lewis' work is that the Thom spectrum
functor $M$ ``commutes'' with these monads.  Specifically, Lewis
proves \cite[7.7.1]{lms}  

\begin{thm}
\hspace{5 pt}
\begin{enumerate}
\item{Given a map $f \colon X \rightarrow BF$, there is an isomorphism
$\tilde{\bC} Mf \cong M(\bK_{BG} f)$, where the map 
\[\bK_{BG}f \colon \bK_{BG} X \rightarrow BG\]
is the natural map induced from $X \rightarrow BG$.} 
\item{This isomorphism is coherently compatible with the unit and
  multiplication maps for these monads, in the sense of lemma
  \ref{l:monad-comm}.}
\end{enumerate}
\end{thm}

As we have observed, a consequence of this result is that the Thom
spectrum functor induces a functor $M_{E_\infty}$ from $(\aT /
BG)[\bK_{BG}]$ to $E_\infty$-ring spectra structured by
$\tilde{\bC}$.  Composing with the functor $S \sma_{\sL} -$, we obtain
a Thom spectrum functor $M_{\aC\aA_S}$ from $(\aT / BG)[\bK_{BG}]$ to
commutative $S$-algebras.  Now employing theorem \ref{lifting}, we
obtain the main result.

\begin{thm}
The Thom spectrum functor 
\[M_{\aC\aA_S} \colon (\aT / BG)[\bL_{BG}] \to \aC\aA_S\]
commutes with indexed colimits.
\end{thm}

\begin{proof}
We have verified that the functor $M_{E_\infty}$ satisfies the
hypotheses of theorem \ref{lifting}, and so we can conclude that
$M_{E_\infty}$ commutes with indexed colimits.  Since $M_{\aC\aA_S}$ is
obtained from $M_{E_\infty}$ via composition with a continuous left
adjoint, the result follows.
\end{proof}

Since the Thom spectrum functor $M_{\aC\aA_S}$ preserves indexed
colimits, one would expect that it should in fact be a continuous left
adjoint.  We will prove this by showing that the hypotheses of the
second part of theorem \ref{lifting} are satisfied.  However, our
method of proof does not produce an explicit description of the right
adjoint and so is somewhat unsatisfying.

\begin{lem}
The Thom spectrum functor from $\aT / BG$ to $\aS \backslash S$ is a
left adjoint. 
\end{lem}

\begin{proof}
We know that the Thom spectrum functor preserves colimits in $\aT /
BG$.  Moreover, it is easy to verify that the category $\aT / BG$
satisfies the hypotheses of the special adjoint functor theorem, since
$\aT$ does.  Therefore $M$ is a left adjoint. 
\end{proof}

Now, we have the following diagram of categories :
\[
\xymatrix{
\relax\aT / BG[\bK_{BG}] \ar@<-0.5ex>[d]_U \ar[r]^-{M_{E_\infty}} & \relax(\aS \backslash S)[\tilde{\bC}] \ar@<-0.5ex>[d]_V \\
\relax\aT / BG \ar@<-0.5ex>[u]_F \ar[r]^M & \relax(\aS \backslash S) \ar@<-0.5ex>[u]_G. \\
}
\]
Here $U$ and $V$ denote forgetful functors, and $F$ and $G$ denote the
free algebra functors.  Recall that $(\aS \backslash S)[\tilde{\bC}]$
is the category of $E_\infty$-ring spectra \cite[2.4.5]{ekmm}.  The
square commutes in the sense that $M \circ U = V \circ M_{E_\infty}$
and $M_{E_\infty} \circ F = G \circ M$.

\begin{cor}
The Thom spectrum functor $M_{\aC\aA_S}$ from $\aT / BG[\bK_{BG}]$ to
the category of commutative $S$-algebras is a continuous left adjoint.  
\end{cor}

\begin{proof}
It follows from theorem \ref{lifting} that $M_{E_\infty}$ is a
continuous left adjoint.  Since $S \sma_{\sL} -$ is a continuous left
adjoint, the composite functor to commutative $S$-algebras is a
continuous left adjoint as well. 
\end{proof}

When restricting attention to vector bundles, we can refine this
result somewhat.  Recall that the categories of $\sL$-spaces,
$E_\infty$-ring spectra, and commutative $S$-algebras are all
categories of algebras over monads.  In each case, a model structure
is constructed by lifting a cofibrantly generated model structure from
the base category.  As a consequence, we have an explicit description
of the cell objects.  

In each case, the cell objects are colimits of pushouts of the form
\[
\xymatrix{
\bZ A \ar[d]\ar[r] & \ar[d] X_{n-1} \\
\bZ CA \ar[r] & X_n \\ 
}
\]
where $\bZ$ is the appropriate monad and where $A$ to $CA$ is a
generating cofibration in the base category.  For instance, in the
case of $\sL$-spaces, $A \rightarrow CA$ is a map of the form
\[\bigvee_i S^{n_i-1}_+ \rightarrow \bigvee_i D^{n_i}_+.\]
For the category of commutative $S$-algebras, $A \rightarrow CA$ is a
map of the form
\[\bigvee_i \Sigma^\infty S^{n_i-1}_+ \rightarrow \bigvee_i
\Sigma^\infty D^{n_i}_+\]
where here the suspension spectrum functor takes values in
$S$-modules.  The description for $E_\infty$-ring spectra is
analogous.

\begin{cor}
Let $G$ be a group $\sI$-monoid.  Then the functor $M_{\aC\aA_S}$ is a
Quillen left adjoint.
\end{cor}

\begin{proof}
In these cases all maps are good, and so $M$ preserves weak
equivalences.  Therefore, it will suffice to show that $M$ takes
the generating cofibrations and generating acyclic cofibrations to
cofibrations.  The generating cofibrations in $\aT[\bK_{BG}]$ are maps
of the form $h \colon \bK_{BG} A \rightarrow \bK_{BG} CA$, where $A$ is a
wedge of $S^{n_i-1}_+$ and $CA$ the corresponding wedge of $D^n_+$.
The maps from $D^n_+ \rightarrow BG$ is arbitrary, and these choices
determines the maps $S^{n_i-1} \rightarrow BG$.  Denote the map
$\bK_{BG} A \rightarrow BG$ by $h_1$ and the map $\bK_{BG} CA
\rightarrow BG$ by $h_2$.  Recall that $M \bK_{BG} f \cong \tilde{\bC}
Mf$.  In addition, a map from a contractible space to $BG$ represents
a bundle which is isomorphic to a trivial bundle.  Therefore, there is
a homeomorphism $Mh_1 \cong \tilde{\bC} \Sigma^\infty A$ and $Mh_2
\cong \tilde{\bC} \Sigma^\infty CA$.  The induced map $Mh \colon Mh_1
\rightarrow Mh_2$ clearly yields a generating cofibration in the
category of $E_\infty$-ring spectra structured by $\tilde{\bC}$.  The
analysis for the acyclic generating cofibrations is similar. 
\end{proof}

\section{Computing $THH$}\label{sec:compute}

The formula $M(f \otimes S^1) \cong Mf \otimes S^1$ is a point-set
result --- $Mf \otimes S^1$ is an object in the category of
commutative $S$-algebras.  In this section we discuss how to ensure
that $Mf \otimes S^1$ has the correct homotopy type so that it
represents $THH(Mf)$.  

For an $S$-algebra $R$, in analogy with the classical definition of
Hochschild homology as $\tor$ we define \[THH(R) = R \sma^{L}_{R \sma
R^{\op}} R.\]
In the algebraic setting, this $\tor$ can be computed via the
Hochschild resolution.  In spectra, this leads to a candidate
point-set description of $THH(R)$ as the cyclic bar construction
$N^{\cyc}(R)$.  The precise relationship between these is studied in
\cite[9.2]{ekmm}; the main result is that when $R$ is cofibrant they
are canonically isomorphic in the derived category of $R$-modules
\cite[9.2.2]{ekmm}.

First, observe that there is a derived version of the cyclic bar
construction in $\sL$-spaces.  This is a consequence of the very
useful fact that for a simplicial set $A_\cdot$ and an $\sL$-space
$X$, there is a homeomorphism $X \otimes |A_\cdot| \cong |X \otimes
A_\cdot|$ \cite[6.7]{basterra-mandell}.  When $A_\cdot$ has finitely many
nondegenerate simplices in each simplicial degree, this provides a
tractable description of the tensor with $|A_\cdot|$ in terms of
tensors with finite sets --- i.e., finite coproducts.

\begin{lem}
Let $g \colon X \rightarrow X^{\prime}$ be a weak equivalence of cofibrant
$\sL$-spaces.  Then there is an induced weak equivalence $g \otimes
S^1_+ \colon X \otimes S^1_+ \rightarrow X^{\prime} \otimes S^1_+$.
\end{lem}

\begin{proof}
Since $X \otimes S^1_+$ is a proper simplicial space for any
$\sL$-space $X$, the result follows from the fact that the induced map
$g \coprod g \colon X \coprod X \rightarrow X^{\prime} \coprod X^{\prime}$
is a weak equivalence when $X$ and $X^{\prime}$ are cofibrant.
\end{proof}

One might hope that for cofibrant $X$, $Mf$ is necessarily cofibrant
as a (commutative) $S$-algebra.  Of course when $M$ is a left Quillen
adjoint this holds, but in general it turns out that $Mf$ does belong
to a class of commutative $S$-algebras for which the point-set smash
product has the correct homotopy type.

\begin{thm}
\item{Let $f \colon X \rightarrow BG$ be a good $\sL$-map such that $X$ is
  a cell $\sL$-space.  Then $Mf \sma Mf$ represents the derived smash
  product.}
\end{thm}

Recall the notion of an extended cell module \cite[9.6]{basterra}.  An
extended $S$-cell is a pair $(X \sma B^n_+, X \sma S^{n-1}_+)$, where
$X = S \sma_{\sL} \sL(i) \thp_G K$ for a $G$-spectrum $K$ indexed on
$U^i$ which has the homotopy type of a $G$-CW-spectrum for some $G
\subset \Sigma^i$.  An extended cell $S$-module is an $S$-module $M =
\colim M_i$ where $M_0 = 0$ and $M_n$ is obtained from $M_{n-1}$ by a
pushout of $S$-modules of the form

\[
\xymatrix{
\bigvee_j X_j \sma S^{n_j-1}_+ \ar[d]\ar[r] & \ar[d] M_{n-1} \\
\bigvee_j X_j \sma B^{n_j}_+ \ar[r] & M_n. \\
}
\]

Extended cell $S$-modules have the correct homotopy type for the
purposes of the smash product.  Therefore, it will suffice to show the
following result.

\begin{prop}
Let $f \colon X \rightarrow BG$ be a good $E_\infty$-map over the linear 
isometries operad such that $X$ is a cell $\sL$-space.  Then the
underlying $S$-module of the $S$-algebra $Mf$ has the homotopy type of
an extended cell $S$-module.
\end{prop}

\begin{proof}
By hypothesis, $X = \colim X_i$ where $X_0 = *$ and $X_i$ is obtained
from $X_{i-1}$ as the pushout 
\[
\xymatrix{
\tilde{\bK}A \ar[r]\ar[d] & \ar[d] X_{i-1} \\
\tilde{\bK}CA \ar[r] & X_i \\
}
\]
where $A$ is a wedge of spheres $S^{n_i-1}_+$ and $CA$ is the associated
wedge of $D^{n_i}_+$.  Since $M$ commutes with colimits and $M\bK g
\cong \tilde{\bC} Mg$, we have that $Mf = \colim Mf_i$ where $Mf_0 =
S$ and $Mf_i$ is obtained from $Mf_{i-1}$ as the pushout  
\[
\xymatrix{
\tilde{\bC}MA \ar[r]\ar[d] & \ar[d] Mf_{i-1} \\
\tilde{\bC}MCA \ar[r] & Mf_i. \\
}
\]
As $CA$ is a contractible space with a disjoint basepoint, $MCA$ is
homotopy equivalent to a cell $S$-module.  $MA$ is the wedge of a Thom
spectrum over a suspension with $S$, and so we know that it is also a
cell $S$-module \cite[7.3.8]{lms}.  Temporarily assume that $\tilde{\bC}MA$
and $\tilde{\bC}MCA$ are extended cell $S$-modules.  Then we proceed as in
\cite[7.7.5]{ekmm}.  $Mf_i$ is isomorphic under $Mf_{i-1}$ to the
two-sided bar construction $B(\tilde{\bC}MCA, \tilde{\bC}MA,
MX_{i-1})$.  This is a proper simplicial spectrum, and since each
simplicial level is an extended cell module and the face and
degeneracy maps are cellular, so is the bar construction.  By passage
to colimits, the result follows.

To see that $\tilde{\bC}MA$ and $\tilde{\bC}MCA$ are extended cell
$S$-modules, we essentially argue as in \cite[7.7.5]{ekmm} but must
account for the quotients since we are using the reduced monads.
Recall that there is a standard filtration on the reduced monads
\cite[7.3.6]{lms}, which allows us to regard the free $\tilde{\bC}$
algebra as the colimit of spectra formed by pushouts of layers of the
form $Z^j / \Sigma^j$.  These are extended cell $S$-modules, and then
a similar induction as above allows us to conclude the
result.
\end{proof}

There is an additional difficulty that arises when working over $BF$;
it seems to be difficult to replace an arbitrary map of $\sL$-spaces
$X \to BF$ with a map $X' \to BF$ which is a Hurewicz fibration and
such that $X'$ is cofibrant as an $\sL$-space.  However, it turns out
to suffice to work with the following composite replacement --- given
an arbitrary map of $\sL$-spaces $X \to BF$, we work with $\Gamma
X' \to BF$,  where $X'$ is a cofibrant replacement of $X$.  For a
detail analysis of this situation, we refer the reader to the
companion paper \cite{blumberg-cohen-schlichtkrull}, as it depends on
a description of $\sL$-spaces as commutative monoids with respect to a
product on the category of $\sL(1)$-spaces defined in analogy with the
EKMM smash product.

%\begin{thm}
%Let $f \colon X \to BF$ be a map of $\sL$-spaces such that $X$ is
%cofibrant.  There is an induced map $f \vee f \colon X \vee X \to
%X \to BF$; the first map in the composite is the fold map.  Then the
%Thom spectrum of the map $\Gamma (f \vee  f)$ represents the
%derived smash product $M\Gamma f \sma^{L} M\Gamma f$.  
%\end{thm}

%Since $\Gamma$ commutes with geometric realization \cite[7.1.9]{lms},
%the preceding result implies that $M(\Gamma (f \otimes S^1))$
%represents $THH(\Gamma f)$.

\section{Splitting of $THH(Mf)$}\label{sec:split}

In the previous section, we have verified that by appropriate
modification of the map $f \colon X \rightarrow BG$ we can ensure that we
can identify $THH(Mf)$ as $M(f \otimes S^1)$.  In this section, we
study $M(f \otimes S^1)$.  In particular, we will discuss briefly a
connection to the free loop space $LBX$ and then investigate in detail
the splitting result $THH(Mf) \htp Mf \sma BX_+$. 

The starting point for our analysis is the observation that the based
cofiber sequence $S^0 \rightarrow S^1_+ \rightarrow S^1$ yields an
associated sequence of $\sL$-spaces 
\[X \rightarrow X \otimes S^1_+ \rightarrow X \otimes S^1.\]
The map $X \rightarrow X \otimes S^1_+$ is split by the collapse map
$S^1_+ \rightarrow S^0$, and this induces a map $\theta : X \otimes S^1_+
\rightarrow X \times (X \otimes S^1)$.

\begin{rem}
Recall that $X \otimes S^1_+$ is the realization of the simplicial
object $X \otimes (S^1_+)_\bullet$ induced by the standard description
of $S^1_+$ as a simplicial set.  This is in fact a cyclic object, and
therefore $X \otimes S^1_+$ has an action of $S^1$ induced by the
cyclic structure.  The adjoint of the action map composed with the
projection $X \otimes S^1_+ \rightarrow X \otimes S^1$ yields a map $X
\otimes S^1_+ \rightarrow L(X \otimes S^1)$ which is a weak
equivalence for group-like $\sL$-spaces.  When working over a group
$\sI$-FCP, this weak equivalence implies a weak equivalence of Thom
spectra, and so we obtain a description of $THH(Mf)$ in terms of a map
$L(BX) \rightarrow BG$.  This relationship is studied in detail in the
companion paper \cite{blumberg-cohen-schlichtkrull}, and we do not
discuss it further herein.
\end{rem}

\subsection{The splitting arising from an $E_\infty$-map}

In this section, we will assume we have a fixed $\sL$-map $f \colon X
\rightarrow BG$ such that $X$ is a group-like $\sL$-space and $G$ is a
group $\sI$-FCP.  We require this latter hypothesis to ensure that all
maps that arise are good.

\begin{lem}
Let $X$ be a group-like cofibrant $\sL$-space.  The map $\theta : X \otimes
S^1_+ \rightarrow X \otimes S^1 \times X \otimes S^0$ is a weak
equivalence.
\end{lem}

\begin{proof}
Since $\sL$ is an $E_\infty$-operad, we can functorially associate an
$\Omega$-prespectrum $Z$ to $X$ using an ``infinite loop space
machine''.  We will show that that $X \otimes A$ is weakly equivalent
to $\Omega^\infty (Z \sma A)$.  Assuming this fact, the lemma is now a
consequence of the stable splitting of $S^1_+$.  Specifically, there
is a chain of equivalences $Z \sma S^1_+ \htp (Z \sma S^0) \vee (Z
\sma S^1) \htp (Z \sma S^0) \times (Z \sma S^1)$.  Applying
$\Omega^\infty$ to this composite yields an equivalence $\Omega^\infty
(Z \sma S^1_+) \rightarrow (\Omega^\infty Z) \times (\Omega^\infty (Z \sma
S^1))$, since $\Omega^\infty$ preserves products and weak equivalences
of spectra.  Under the equivalence between $X$ and $\Omega^\infty Z$,
this map coincides with the map induced from the splitting and so the
result follows.

To compare $X \otimes A$ and $Z \sma A$, we use a technique from
\cite{basterra-mandell}.  Let $\tilde{X}$ denote the functor which
assigns to a finite set $\underbar{n}$ the tensor $X \otimes
\underbar{n}$.  Using the folding map, this specifies a
$\Gamma$-object in $\sL$-spaces.  Recall that the construction of a
prespectrum from a $\Gamma$-object proceeds by prolonging the
$\Gamma$-object to a functor from the category of spaces of the
homotopy type of finite $CW$-complexes.  Such a functor is called a
$\aW$-space, and is an example of a diagram spectrum \cite{mmss}.  In
this situation, the associated $\aW$-space can be specified simply as
$A \mapsto X \otimes A$.  For any $\aW$-space $Y$ and based space $A$,
there is a stable equivalence between the prespectrum $\{Y(S^n) \sma
A\}$ and the prespectrum $\{Y(A \sma S^n)\}$ induced by the assembly map
$Y(S^n) \sma A \rightarrow Y(A \sma S^n)$ \cite[17.6]{mmss}.  Since
$X$ was a cofibrant group-like $\sL$-space, $\tilde{X}$ is very special
\cite[6.8]{basterra-mandell}. Therefore the associated $\aW$-space
$\tilde{X}$ is fibrant, which means that the underlying prespectra
$\{\tilde{X}(S^n \sma A)\}$ are $\Omega$-prespectra for all $A$.
Finally, this implies that there is an equivalence between
$\Omega^\infty (Z \sma A)$ and $Z(A)$.  A similar result (with a
different proof) appears in \cite{schlichtkrull}.
\end{proof}

\begin{prop}
Let $f \colon X \rightarrow BG$ be an $\sL$-map where $G$ is a group
$\sI$-FCP and $X$ is a cofibrant group-like $\sL$-space.  Then there
is a weak equivalence of commutative $S$-algebras
\[Mf \otimes S^1 \simeq BX_+ \sma Mf.\]
\end{prop}

\begin{proof}
By inspection of the description of the map 
$f \otimes S^1_+ : X \otimes S^1_+ \rightarrow BG$, we see that it can
be factored  
\[
\xymatrix{
X \otimes S^1_+ \ar[r]^-\theta & (X \otimes S^0) \times (X \otimes S^1)
\ar[r]^-{\pi_1} & X \otimes S^0 \cong X \ar[r]^-f & BG,
}
\] 
where $\pi_1$ is the projection onto the first factor.  By the
preceding lemma, the hypotheses imply that the map $\theta : X \otimes
S^1_+ \rightarrow (X \otimes S^1) \times (X \otimes S^0)$ is a weak
equivalence.  Therefore, there is an equivalence of Thom spectra
$M\theta : M(f \otimes S^1_+) \rightarrow M(f \circ \pi_1)$.  By the
standard description of the Thom spectrum of a projection (proposition
\ref{P:ten}), we know that $M(f \circ \pi_1) \cong Mf \sma (X \otimes
S^1)_+$.  Moreover, theorem \ref{commutation} implies that $M(f
\otimes S^1_+) \cong Mf \otimes S^1$.  Finally, $X \otimes S^1$ is a
model of $BX$ --- this follows by considering the $\Gamma$-space
associated to $X$ as in the previous lemma
\cite[6.5]{basterra-mandell}. 
\end{proof}

\subsection{Splitting arising from an $E_2$-map $f \colon X \rightarrow BG$}

It is sometimes the case that even though $f \colon X \rightarrow BG$ is
not an $E_\infty$-map, $Mf$ is equivalent to a commutative
$S$-algebra.  We will consider the situation in which $f \colon X
\rightarrow BG$ is an $E_2$-map such that the there is an equivalence
of $E_2$-ring spectra from $Mf$ to an $E_\infty$-ring spectrum.
Although this may seem at first like an artificial hypothesis, in fact
this situation arises when considering the Thom spectra that yield
Eilenberg-Mac Lane spectra.  We will show that the splitting result
holds here as well.

Fix an $E_2$-operad $\aC_2$ which is augmented over the linear
isometries operad.  Then $BG$ is a $\aC_2$-space and Lewis' theorem
\cite[7.7.1]{lms} shows that the Thom spectrum associated to an
$\aC_2$-map $f \colon X \rightarrow BG$ is a $\aC_2$-ring spectrum.

Recall that there is a two-sided bar construction for spectra
\cite[4.7.2]{ekmm}.  Let $R$ be a commutative $S$-algebras.  If $A$ is
a left $R$-module and $N$ a right $R$-module, the bar construction
$B(A,R,N)$ is the realization of a simplicial spectrum in which the
$k$-simplices are given by $A \sma R^k \sma N$ and the faces are given
by the multiplication.  When $R$ is a cofibrant commutative
$S$-algebra and $A$ is a cofibrant $R$-module, the bar construction is
naturally weakly equivalent to $A \sma_R N$ and weak equivalences in
each variable induce weak equivalences of bar constructions.

\begin{rem}
A simplicial spectrum $K$ is proper if the ``inclusion'' $sK_q
\rightarrow K_q$ is a cofibration, where $sK_q$ is the ``union'' of
the subspectra $s_j K_{q-1}, 0 \leq j < q$ \cite[10.2.2]{ekmm}.  Of
course, the ``union'' denotes an appropriate pushout, and the ``inclusion''
associated maps, but the terms are useful to emphasize the analogy
with the situation in spaces.  Maps between proper simplicial spectra
which induce levelwise equivalences produce weak equivalences upon
realization \cite[10.2.4]{ekmm}.  When $R$ is a cofibrant commutative
$S$-algebra and $A$ is a cofibrant $R$-module, the bar construction is
a proper simplicial spectrum. 
\end{rem}

\begin{thm}\label{T:e2split}
Let $f \colon X \rightarrow BG$ be a good $\aC_2$-map.  Assume that $Mf$ is
equivalent as a homotopy commutative $S$-algebra to some (strictly)
commutative $S$-algebra $M^{\prime}$.  Then there is an isomorphism in
the derived category 
\[THH(Mf) \simeq BX_+ \sma Mf.\]
\end{thm}

\begin{proof}
$THH(A)$ can be described as the derived smash product $A
\sma^L_{A \sma A^{\op}} A$ \cite[9.1.1]{ekmm}.  Of course if $A$ is
commutative, $A \sma A^{\op} \cong A \sma A$.  In our situation, this
specializes to the derived smash product 
\[THH(Mf) = Mf \sma^L_{Mf \sma Mf^{\op}} Mf.\]
If $Mf$ were a commutative $S$-algebra, we could use the Thom
isomorphism to replace $Mf \sma Mf^{\op} \cong Mf \sma Mf$.  We will
show that in fact it suffices for $Mf$ to be weakly equivalent to a
commutative $S$-algebra.  We can assume without loss of generality
that $Mf$ is cofibrant.  Moreover, the hypotheses provide an
equivalence of $S$-algebras $Mf \rightarrow M^\prime$, where
$M^\prime$ can be taken to be a cofibrant commutative $S$-algebra.

The composite \[Mf^{\op} \rightarrow Mf^{\op} \sma S_0 \rightarrow
Mf^{\op} \sma X^{\op}_+ \rightarrow (M^{\prime})^{\op} \sma X^{\op}_+
\htp M^{\prime} \sma X^{\op}_+\] is a map of $S$-algebras, and the map
$M^{\prime} \rightarrow M^{\prime} \sma S^0 \rightarrow M^{\prime}
\sma X^{\op}_+$ is central \cite[7.1.2]{ekmm}.  Therefore extension of
scalars yields an induced map of $M^{\prime}$-algebras $M^{\prime}
\sma Mf^{\op} \rightarrow M^{\prime} \sma X^{\op}_+$, and the Thom
isomorphism theorem implies this map is a weak equivalence.

We will model the derived smash product using the two-sided bar
construction.  The preceding discussion implies that the composite
\[B(Mf, Mf \sma Mf^{\op}, Mf) \rightarrow B(M^\prime, Mf \sma Mf^{\op},
M^{\prime}) \rightarrow B(M^{\prime}, M^{\prime} \sma X^{\op}_+,
M^{\prime})\] is a weak equivalence.  Therefore we have an isomorphism 
\[Mf \sma^L_{Mf \sma Mf^{\op}} Mf \rightarrow M^\prime
\sma^L_{M^\prime \sma X^{\op}_+} M^\prime.\]

The $k$th simplicial level of $B(M^{\prime}, M^{\prime} \sma
X^{\op}_+, M^{\prime})$ is the product \[M^{\prime} \sma (M^{\prime}
\sma X^{\op}_+)^k \sma M^{\prime},\] where the actions of $M^{\prime}
\sma X^{\op}_+$ on $M^{\prime}$ are given by projecting $M^{\prime}
\sma X^{\op}_+ \rightarrow M^{\prime}$ and then using the
multiplication on $M^{\prime}$.  Clearly, there is an isomorphism
\[M^{\prime} \sma (M^{\prime} \sma X^{\op}_+)^k \sma M^{\prime}
\rightarrow (M^{\prime} \sma (M^{\prime})^k \sma M^{\prime} \sma
(X^{\op}_+)^k\]
given by permuting the $X^{\op}_+$ factors to the right, and this map
commutes with the simplicial identities.  Thus, there is an equivalence
\[B(M^{\prime},M^{\prime} \sma X^{\op}_+, M^{\prime}) \htp
B(M^{\prime},M^{\prime},M^{\prime}) \sma B(S, \Sigma^\infty X^{\op}_+,
S), \] 
using the fact that the smash product commutes with realization.
However, since $\Sigma^\infty$ commutes with the bar construction for
monoids \cite{ekmm}, we have weak equivalences \[B(S, \Sigma^\infty
X^{\op}_+, S) \htp \Sigma^\infty BX^{\op}_+ \htp \Sigma^\infty BX_+.\]
We also know that $B(M^{\prime}, M^{\prime}, M^{\prime})$ is homotopic
to $M^{\prime}$.
\end{proof}

Notice that the preceding proof did not require $X$ to be a cofibrant
$\sL$-space, and so we can circumvent issues of the interaction of
$\Gamma$ and cofibrant replacement in the applications.

\section{Calculation of $THH(\Z)$, $THH(\Z/p)$, and $THH(MU)$}

In this section, we use the splitting results of the previous section
to provide easy calculations of $THH$ for various interesting Thom
spectra.  First, we recover results of Bokstedt for $H\Z/p$ and
$H\Z$ \cite{bokstedt2}.  Next, we compute $THH(MU)$, recovering a
calculation of McClure and Staffeldt \cite{mcclure-staffeldt}.
Further calculations of bordism spectra are discussed in the companion
paper \cite{blumberg-cohen-schlichtkrull}.

\subsection{$THH(\Z)$ and $THH(\Z/p)$}

There is an identification due to Mahowald of $H\Z/2$ as the Thom
spectrum associated to a certain map $\Omega^2 S^3 \rightarrow BO$
\cite{cohen-may-taylor, mahowald}.  A modification of this approach
due to Hopkins allows the construction of $H\Z/p$ as the Thom spectrum
associated to a certain $p$-local bundle over $\Omega^2 S^3$.
Finally, $H\Z$ can be obtained as the Thom spectrum of a map $\Omega^2
S^3 \left<3\right> \rightarrow BSF$.  We will discuss these
constructions in the following section, in particular verifying that
all of these Thom spectra are $E_2$ ring spectra associated to
$E_2$ maps structured by the little 2-cubes operad.  Using standard
``change of operad'' techniques discussed in Appendix~\ref{appop}, we
can functorially convert these to classifying maps structured by an
$E_2$ operad augmented over the linear isometries operad.

We have the following proposition, which will allow us to apply
theorem \ref{T:e2split}.

\begin{prop}
For any connective $E_2$-ring spectrum $R$, there is a map of
$E_2$-ring spectra from $R$ to $H\pi_0(R)$, unique up to homotopy,
which induces an isomorphism on $\pi_0$.  Here $H\pi_0(R)$ is regarded 
as an $E_2$-ring spectrum by forgetting from the commutative
$S$-algebra structure.
\end{prop}

Recall that $THH(HR)$ for $R$ a commutative ring is a product of
Eilenberg-Mac Lane spectra \cite{bokstedt2}, \cite[9.1.3]{ekmm}.  This
implies that we can read off the homotopy type from the homotopy
groups.  Thus to compute $THH(H\Z/2)$, we must compute
$\pi_*(B(\Omega^2 S^3) \sma H\Z/2)$.  This is just the homology of
$\Omega S^3$ with $\Z/2$ coefficients, which can be easily calculated
via inspection of the James construction.  One easily recovers the
result
\[THH(H\Z/2) = \prod_{i=0}^{\infty} K(\Z/2, 2i).\] 
A similar argument applies to $THH(H\Z/p)$.

Finally, to compute $THH(H\Z)$, we must compute $\pi_*(B(\Omega^2
S^3\left<3\right>) \sma H\Z)$.  Once more, this is just the ordinary
homology with integral coefficients of $\Omega S^3\left<3\right>$.
Computing again, we find 
\[THH(H\Z) = K(\Z,0) \times \prod_{i=1}^\infty K(\Z/i, 2i-1).\]

\subsection{THH(MU)}

The splitting formula implies that 
\[THH(MU) \htp MU \sma BBU_+ \htp MU \sma SU_+.\]
We can compute $MU_*(SU)$ via a standard Atiyah-Hirzebruch spectral
sequence calculation, and it turns out to be $MU_*(pt) \otimes
\Lambda(x_1, x_2, \ldots)$, with the generators in odd degrees.  This
agrees with the answer obtained by McClure and Staffeldt
\cite{mcclure-staffeldt}, and as they observe implies that $THH(MU)$
is a product of suspensions of $MU$.  Other bordism spectra are
analogous; see the companion paper \cite{blumberg-cohen-schlichtkrull}
for further discussion.

\section{Realizing Eilenberg-Mac Lane spectra as Thom spectra}\label{S:TEM}

In this section, we review and extend the classical realizations of
Eilenberg-Mac Lane spectra as Thom spectra associated to certain
bundles over $\Omega^2 S^3$ and $\Omega^2 S^3 \left<3\right>$.  Our
main purpose is to ensure that we can obtain these Thom spectra as
ring spectra which are sufficiently structured so as to permit the
construction of $THH$ and the application of our splitting theorem.
In particular, improving on \cite{cohen-may-taylor} we give a new
description of $H\Z$, based on a suggestion of Mike Mandell, as the
Thom spectrum associated to a double loop map $\Omega^2
S^3\left<3\right> \rightarrow BSF$.

\subsection{$H\Z/2$ as the Thom spectrum of a double loop map}

The construction of $H\Z/2$ as a Thom spectrum was the first to be
extensively studied \cite{cohen-may-taylor, mahowald, priddy}.  We
briefly review the construction.  Consider the map $\psi \colon S^1
\rightarrow BO$ representing the nontrivial element of $\pi_1(BO)$.
The Thom spectrum associated to this map is the Moore spectrum
$M\Z/2$.  There is an induced map $\gamma \colon \Omega^2 S^3 \rightarrow
BO$, as $BO$ is an infinite loop space (and in particular a double
loop space).  The Thom spectrum of $\gamma$ is $H\Z/2$.

A sketch of the proof for this is as follows.  There is a map $\aA
\rightarrow H^*(M\gamma;\Z/2)$ given by evaluation on the Thom class
which is a map of modules over the Steenrod algebra.  As $M\gamma$ is
$2$-local, it suffices to show that this map is an isomorphism.
Dualizing, we can consider the corresponding map $H_*(M\gamma; \Z/2)
\rightarrow \aA^*$ of comodules over the dual Steenrod algebra
$\aA^*$.  Next, by the Thom isomorphism we know that
$H_*(M\gamma;\Z/2) \cong H_*(\Omega^2 S^3; \Z/2)$.  The homology of
$\Omega^2 S^3$ is $P\{x_n \mid n \geq 0\}$, where $x_0$ comes from the
inclusion of $H_*(S^1; \Z/2)$ and the action of the Dyer-Lashof
operations is known \cite{cohen-may-taylor} --- specifically, $x_0$
generates the homology as a module over the Dyer-Lashof algebra.  Now,
note that since the dimensions of $\aA$ and $H^*(\Omega^2 S^3; \Z/2)$
are the same, it is enough to show that the evaluation map is either
an injection or a surjection.  

There are a variety of arguments to establish this fact; we will
review the technique used by \cite{priddy}.  First, we observe that
both the Thom isomorphism and the map $\gamma_* \colon H_*(\Omega^2 S^3; \Z/2)
\rightarrow H_*(BO; \Z/2)$ commute with the Dyer-Lashof operations.
Recall that $H_*(BO; \Z/2)$ is generated by the images of the class in
degree 1 under the first Dyer-Lashof operation.  Therefore the
behavior of $\gamma_*$ is completely determined by the fact that
$\gamma_* (x_0)$ is that generating class in degree 1.  Finally, we
note that under the evaluation map $H_*(MO; \Z/2) \rightarrow \aA$ the
images of the iterates of $\gamma_*(x_0)$ under the Dyer-Lashof
operation hit all of the generators of $\aA$.

\subsection{$H\Z/p$ as the Thom spectrum of a double loop map}

Unfortunately, no stable spherical fibration can have $H\Z/p$ as its
associated Thom spectrum --- $\pi_0(Mf)$ is either $\Z$ or $\Z/2$,
depending on whether $f$ represents an orientable bundle or not.
Nonetheless, in \cite{thomified} there is a brief discussion of an
argument due to Hopkins for realizing $H\Z/p$ as the Thom spectrum
associated to a $p$-local stable spherical fibration.

In the bulk of this paper, we studied Thom spectra associated to
monoid $\sI$-FCP's which were augmented over $F$.  The map to $X
\rightarrow F$ was used to give an action of $X(V)$ on the sphere
$S^V$, the fiber of the universal quasifibration $B(*,X(V),S^V)
\rightarrow B(*,X(V),*)$.  However, as we noted previously, this
theory can be carried out with other choices of fiber, in particular
the collection of $p$-local spheres $S^V_{(p)}$ or $p$-complete
spheres $(S^V)\phat$.  Rather than an augmentation over $F$, we will
in this setting require augmentation over the appropriate ``p-local''
or ``p-complete'' analogue.  We rely on the careful treatment of
fiberwise localization and completion given by May
\cite{may-fiberloc}.

\begin{defn}
\hspace{5 pt}
\begin{enumerate}
\item{Let $F_{(p)}$ denote the monoid $\sI$-FCP specified by taking
  $V$ to the based homotopy self-equivalences of $S^V_{(p)}$.  Denote
  by $BF_{(p)}$ the $\sI$-FCP obtained by passing to classifying
  spaces levelwise.}
\item{Let $(F)\phat$ denote the $\sI$-FCP specified by taking $V$ to
  the based homotopy self-equivalences of $(S^V)\phat$.  Denote by
  $B(F)\phat$ the $\sI$-FCP obtained by passing to classifying spaces
  levelwise.}
\end{enumerate}
\end{defn}

$BF_{(p)}(V)$ classifies spherical fibrations with
fiber $S^V_{(p)}$ and $B(F)\phat(V)$ classifies spherical fibrations
with fiber $(S^V)\phat$ \cite{may-fiberloc}.  Note that we must use
continuous versions of localization and completion in order to ensure
we have continuous functors \cite{iwase}.

\begin{rem}
The notation we are using is potentially confusing, as the spaces
$BF_{(p)}(V)$ are not the p-localizations of $BF(V)$ and the spaces
$B(F)\phat$ are not the p-completions of $BF(V)$.  Such equivalences
are only true after passage to universal covers, as there is an
evident difference at $\pi_1$.
\end{rem}

In this setting, we can set up the theory of Thom spectra as discussed
in previous sections of the paper with minimal modifications.  For
oriented bundles, there is a Thom isomorphism with $\Z_{(p)}$ or
$Z\phat$ respectively and for unoriented bundles there is a
$\Z/p$ Thom isomorphism \cite{may-fiberloc}.

Now, $\pi_1(BF_{(p)})$ is the group of $p$-local units $\Z_{p}^{\times}$.
Consider a map $\phi \colon S^1 \rightarrow BF_{p}$ associated to a
choice of unit $u$.  The Thom spectrum associated to
$\phi$ is the Moore spectrum obtained as the cofiber of the map $S_{p}
\rightarrow S_{p}$ given by multiplication by $u-1$.  This
identification follows immediately from the general description of the
Thom spectrum of a bundle over a suspension \cite[9.3.8]{lms}.  Taking
$u = p+1$, which is a $p$-local unit, we obtain the Moore spectrum
$M(\Z/p)$.  As before, there is an induced map $\gamma \colon \Omega^2 S^3
\rightarrow BF_{(p)}$ since $BF_{(p)}$ in an infinite loop space.  

We will show that the Thom spectrum associated to this map is $H\Z/p$.
Once again, the Thom class specifies a map $\aA_{p} \rightarrow
H^*(M\gamma)$ of modules over the Steenrod algebra.  For odd $p$,
$H_*(\Omega^2 S^3;\Z/p) = E\{x_n \mid n \geq 0\} \otimes P\{\beta x_n
\mid n \geq 1\}$, where $x_0$ comes from the inclusion of $H_*(S^1;
\Z/p)$, and is generated as a module over the Dyer-Lashof algebra by
$x_0$ \cite{cohen-may-taylor}.  Again, note that since the dimensions
of $\aA$ and $H^*(\Omega^2 S^3; \Z/p)$ are the same, it is enough to
show that the evaluation map is either an injection or a surjection.
This can be shown by an argument analogous to the one described for
$p=2$ above.

\subsection{$H\Z$ as the Thom spectrum of a double loop map}

Finally, we consider the case of $H\Z$.  It has long been known that
$H\Z$ arises as the Thom spectrum associated to a certain map $\gamma
\colon \Omega^2 (S^3\left<3\right>) \rightarrow BSF$
\cite{cohen-may-taylor, mahowald}.  However, the best published
results obtain a description of this map as an $H$-map
\cite{cohen-may-taylor}, which is inadequate for construction of
$THH$.  Moreover, it is not clear how to adapt the existing
construction to improve this --- the map $\gamma$ is constructed a
prime at a time, and the localized maps $\gamma_p$ are seen to be
$H$-maps because certain obstructions vanish.

Therefore, we give a new construction, based on a suggestion of Mike
Mandell, which enables us to see that there is a suitable map which is
a double loop map.  Both 
$\Omega^2 S^3 \left<3\right>$ and $BSF$ are rationally trivial, and so 
split as the product of their completions.  Therefore a map $\Omega^2
S^3 \left<3 \right> \rightarrow BSF$ can be specified by the
construction of a collection of maps $\Omega^2 S^3 \left<3 \right> \rightarrow
(BSF)\phat$.  Note that the $p$-completion of $BSF$ is weakly
equivalent to $\colim_V B((SF)\phat)$, where $(SF)\phat$ is the
monoid $\sI$-FCP constructed analogously to $(F)\phat$.  The
following lemma is standard.

\begin{lem}
Let $f \colon \Omega^2 S^3 \left<3\right> \rightarrow BSF$ be a map
specified by a collection of maps $f_p \colon \Omega^2 S^3 \left<3 \right>
\rightarrow (BSF)\phat$.  If each $f_p$ is an $n$-fold loop map, then
$f$ is an $n$-fold loop map.
\end{lem}

Next, we observe that it will suffice to show that at each prime, the
map given by evaluation on the Thom class induces an equivalence
between the Thom spectrum associated to $\Omega^2 S^3 \left<3
\right> \rightarrow B(SF)\phat$ and $H\Z\phat$.  The Thom class
clearly induces an equivalence in integral homology.  Therefore, if
the evaluation map induces an equivalence in $\Z/p$ cohomology for
each $p$, by naturality it must induce a stable equivalence of
spectra.

For $p = 2$, we can use the map induced by the composite 
\[\Omega^2 S^3 \left<3\right> \rightarrow \Omega^2 S^3 \rightarrow BO
\rightarrow B(SF)\phat.\]
This is a double loop map, and the associated Thom spectrum is
$H\Z\twohat$ \cite{cohen-may-taylor}.  For odd primes, we proceed
as follows.  We know that $\pi_1(B(F)\phat)$ is the group of $p$-adic
units $(\Z\phat)^\times$.  Explicitly, for odd primes this is
$(\Z\phat)^\times \cong \Z/(p-1) \oplus \Z\phat$.
Take a map $\phi$ representing an element of $\pi_1(B(F)\phat)$
which is $0$ on the $\Z/(p-1)$ factor and induces an isomorphism on
the other component.  We can equivalently regard $\phi$ as a map
$\phi \colon S^3 \rightarrow B^3(F)\phat$.  Now, we can lift to a map
$S^3\left<3\right> \rightarrow B^3 (SF)\phat$.  Since $\phi$ is
trivial on the $\Z/(p-1)$ component of $\pi_3(B^3(F)\phat)$, we can
lift the map to the fiber over the map $B^3(F)\phat \rightarrow
K(\Z/(p-1),3)$.  The induced map is an isomorphism on $\pi_3$ by
construction, and so now we can pass to fibers over $K((Z)\phat,3)$
to obtain the desired map.  Looping twice, denote by $\gamma$ the
resulting map $\Omega^2 S^3 \rightarrow B(F)\phat$ and
$\gamma^\prime$ the resulting map $\Omega^2 S^3\left<3\right>
\rightarrow BS(F)\phat$.

First, let us identify the Thom spectrum $M\gamma$.  This proceeds
essentially as in the previous examples.  Specifically, the Thom
spectrum associated to the map $\phi$ is the Moore spectrum obtained
as the cofiber of the map which is multiplication by $u - 1$, where
$u$ is the chosen $p$-adic unit.  This Moore spectrum is determined by
the $p$-adic valuation of $u-1$.  To compute this, let us recall the
identification of the $p$-adic units.  A unit in $(Z)\phat$ is a
$p$-adic integer with an expansion such that the first digit is
nonzero.  The projection onto the units of $\Z/p$ induces the first
component of the identification.  In our case, we are requiring a
choice where the first component is $1$.  Subtracting $1$ from this,
we find that the first component must be $0$ and the later components
are arbitrary.  Combining with the constraint that the projection of
$u$ generates the $(Z)\phat$, we find that we have the Moore spectra
$M(\Z/p)$.  A similar argument to the the one employed above implies
that $M\gamma$ is $H\Z/p$.

Finally, we will use this identification to determine the Thom
spectrum $M\gamma^{\prime}$.  Let us first consider the case of $p$ an
odd prime.  Essentially by construction, there is a commutative
diagram of Thom spectra
\[
\xymatrix{
Mf \ar[r] \ar[d] & M((SF)\phat) \ar[d] \\
H\Z/p \ar[r] & M((F)\phat) \\
}
\]
associated to the commutative diagram of spaces
\[
\xymatrix{
\Omega^2 S^3\left<3\right> \ar[r]\ar[d] & \ar[d] B((SF)\phat) \\
\Omega^2 S^3 \ar[r] & B((F)\phat). \\
}
\]

By the naturality of the Thom isomorphism, this implies that we have a
commutative diagram of modules over the Steenrod algebra
\[
\xymatrix{
\aA \ar[r] \ar[d] & H^*(M\gamma) \ar[d] \\
\aA/ \beta \aA \ar[r] & H^*(M\gamma^\prime) \\
}
\]

The map $\aA \rightarrow \aA / \beta \aA$ is a surjection, we have
seen that the map $\aA \rightarrow H^*(M\gamma)$ is an isomorphism,
and $\Omega^2 S^3 \left<3\right> \rightarrow \Omega^2 S^3$ induces a
surjection on cohomology (and on homology a map of comodules over the
dual Steenrod algebra).  This implies that the top horizontal map must
be a surjection.  Since the dimension of $\aA / \beta \aA$ and
$H^*(\Omega^2 S^3 \left<3\right>; \Z/p)$ are the same, this map must
in fact be an isomorphism.  

\begin{rem}
If we work at the prime $2$, we have that $\pi_1$ is
$(\Z\twohat)^\times = \Z/2 \oplus \Z\twohat$.  Following the outline
above, we would like to identify the Thom spectrum associated to
$\phi$.  The projection onto the units of $\Z/4$ induces the first
component of the identification of $(\Z\twohat)^\times$.  The two
choices are expansions which begin $1, 1, \ldots$ and $1, 0, \ldots$.
Since we want something which projects to $0$, we must have the
latter.  Subtracting $1$ from this, we find we end up with a $p$-adic
number which begins $0, 0, \ldots$ and therefore has $p$-adic
valuation $2$ or higher.  Therefore the associated Thom spectrum is
the Moore spectrum $M(\Z/4)$.

However, consideration of the Dyer-Lashof operations tells us that the
Thom spectrum of $\gamma$ is not $H(\Z/4)$.  In general, we cannot
obtain $H(\Z/p^n)$ as a Thom spectrum over $\Omega^2 S^3$.  This can
be seen by considering the element $x_0$ in $H_1(\Omega^2 S^3)$.  The
last Dyer-Lashof operation takes this to $Q_2 x_0$, but since the
classifying map takes $x_0$ to $0$ it must take $Q_2 x_0$ to zero and
thus must be $0$ on $H^3$ as well, which implies that the Thom
spectrum cannot be the Eilenberg-Mac Lane spectrum.  It is also
possible to deduce the impossibility of realizing $H(\Z/p^n)$ as such
a Thom spectrum by observing that the computations of \cite{brun} are
incompatible with our splitting results.
\end{rem}

\appendix

\section{Change of operads}\label{appop}

The linear isometries operad arises naturally when considering the
infinite loop space structure on $BG$.  Moreover, since we interested
in a Thom spectrum functor which takes values in the EKMM category of
spectra, the presence of the linear isometries operad is to be
expected.  However, it useful to be able to accept a somewhat broader
range of input data.

In some of the examples above, the initial input was maps $X \to B^n
(BF)$, which were looped down to produce $n$-fold loop maps $\Omega^n
X \to \Omega^n B^n (BF)$.  To specify the multiplicative structure
carefully, we need to choose a precise model of the delooping $B$.
Let us assume we are working with a specified choice of $BF$ where the
$E_\infty$ structure is described by an action of the linear
isometries operad $\sL$.  By pullback, we regard this as a space
structured by the product operad $\aC_n \times \sL$, where $\aC_n$ is
the little $n$-cubes operad.  Denote by $\bD$ the monad associated to
this operad.  Following \cite[13.1]{may-geom}, for any $\bD$-space $Z$
we have the diagram
\[
\xymatrix{
Z & \ar[l]_{\htp} B(\bD, \bD, Z) \ar[r]^{\htp} & \Omega^n B(\Sigma^n,
\bD, Z)  
}
\]
in which the maps are maps of $\bD$-spaces, and the action of $\bD$ on
$\Omega^n \Sigma^n$ comes from the augmentation of $\bD$ over the
monad associated to the little n-cubes operad.  The $\bD$-space
action on $\Omega^n B(\Sigma^n, \bD, Z)$ is produced by pullback from
the $\aC_n$ action on $B(\Omega^n \Sigma^n, \bD, Z)$.  Thus, we use
$B(\Sigma^n, \bD, BF)$ as our model of $B^n BF$.  

Given a map $X \to B(\Sigma^n, \bD, BF)$, the associated map $\Omega^n X
\to \Omega^n B(\Sigma^n, \bD, BF)$ is a map of $\bD$-spaces with regard
to the geometric action of the little $n$-cubes operad --- and on
$\Omega^n B(\Sigma^n, \bD, BF)$, this is precisely the action that
arises in the diagram above.  Pulling back, we get a map of
$\bD$-spaces $X' \to B(\bD,\bD,BF)$, and pushing forward along the map
$B(\bD, \bD, BF) \to BF$ we get a map of $\bD$-spaces $X' \to BF$
where the $\bD$ action on $BF$ comes from the augmentation over the
linear isometries operad.

\bibliographystyle{plain}
\bibliography{thom-thesis}

\end{document}